\documentclass{amsart}[12pt]

\usepackage{yfonts} %
\usepackage{amssymb} %
\usepackage{amsthm}
\usepackage{array}
\usepackage{booktabs}%
\usepackage{hhline}%
\usepackage{xy} %
\usepackage{epsfig}%
\usepackage{color}%
\usepackage{upgreek}
\usepackage[english]{babel}
\usepackage{epigraph}%
\usepackage{fancybox}%
\usepackage{shadow}
\usepackage{afterpage}
\usepackage{mathrsfs}
\usepackage{enumitem}
\usepackage{subcaption}
\usepackage{graphicx}
\usepackage{type1cm}
\usepackage{eso-pic}
\usepackage{color}
\usepackage{upgreek}
\usepackage[foot]{amsaddr}

\newtheorem{theorem}{Theorem}
\newtheorem{lemma}{Lemma}
\newtheorem{proposition}{Proposition}
\theoremstyle{definition}
\newtheorem{definition}{Definition}
\newtheorem{remark}{Remark}
\newtheorem{example}{Example}

\theoremstyle{plain}
\newtheorem{corollary}{Corollary}

\newcommand{\vt}{\vspace{.1cm}}

\newcommand{\R}{\mathbb{R} }

\newcommand{\C}{\mathbb{C} }

\newcommand{\h}{\mathbb{H} }
\newcommand{\s}{\mathbb{S} }

\renewcommand{\rho}{\varrho}
\newcommand{\g}{\nabla}

\usepackage{amsmath}

\newcommand{\overbar}[1]{\mkern 1.5mu\overline{\mkern-1.5mu#1\mkern-1.5mu}\mkern 1.5mu}

\newcommand{\ssr}{}
\newcommand{\ssl}{\scriptscriptstyle{L}}
\renewcommand{\theta}{\Theta}
\renewcommand{\tau}{u}

\usepackage{scalerel}[2016-12-29]

\begin{document}

\title{Helicoids and Catenoids in $M\times\R$}
\author{Ronaldo F. de Lima and Pedro Roitman}
\address[A1]{Departamento de Matem\'atica - Universidade Federal do Rio Grande do Norte}
\email{ronaldo@ccet.ufrn.br}
\address[A2]{Departamento de Matem\'atica - Universidade de Brasília}
\email{roitman@mat.unb.br}
\subjclass[2010]{53B25 (primary), 53C24,  53C42 (secondary).}
\keywords{helicoid -- catenoid -- product space.}

\begin{abstract}
Given an arbitrary $C^\infty$ Riemannian manifold $M^n$, we consider the problem
of introducing and constructing minimal hypersurfaces in $M\times\mathbb{R}$ which
have the same fundamental properties of the standard helicoids
and catenoids of Euclidean space $\mathbb{R}^3=\R^2\times\mathbb{R}$. Such hypersurfaces are defined
by imposing  conditions on their height functions and horizontal sections,
and then called  \emph{vertical helicoids} and \emph{vertical catenoids}.
We establish that vertical helicoids in $M\times\mathbb{R}$ have the same
fundamental uniqueness properties of the helicoids in $\mathbb{R}^3.$
We provide several examples of properly embedded vertical helicoids in the case where $M$ is one of the simply connected space
forms. Vertical helicoids which are entire graphs of functions on ${\rm Nil}_3$ and ${\rm Sol}_3$ are also presented.
We show that vertical helicoids of $M\times\R$ whose horizontal sections are totally geodesic in $M$  are locally given by
a "twisting" of a fixed totally geodesic hypersurface of $M.$
We give a local characterization of hypersurfaces of $M\times\mathbb{R}$ which have the gradient of
their height functions as a principal direction. As a consequence, we prove that vertical catenoids exist
in $M\times\mathbb{R}$ if and only if $M$ admits families of isoparametric hypersurfaces.
If so, properly embedded vertical catenoids can be constructed through the solutions of a certain first
order linear differential equation. Finally, we give a complete
classification of the hypersurfaces of $M\times\mathbb{R}$ whose angle function is constant.
\end{abstract}

\maketitle

\section{Introduction}

In this paper, we address the problem of defining and
constructing minimal hypersurfaces in $M\times\R$ with special properties,
where $M^n$ is an arbitrary $C^\infty$ Riemannian manifold.
We will focus our attention on those fundamental properties of the standard helicoids and catenoids
of Euclidean space $\R^3=\R^2\times\R,$
so that the corresponding  minimal hypersurfaces of $M\times\R$ will be called
\emph{vertical helicoids} and \emph{vertical catenoids}.

More specifically, these hypersurfaces will be introduced by imposing conditions on their
horizontal sections (intersections with $M\times\{t\}, \,t\in\R$), and also on the trajectories of the
gradient of their height functions (\emph{height trajectories}, for short).
Vertical helicoids, for instance,  are defined as those hypersurfaces
of $M\times\R$ whose horizontal sections are minimal hypersurfaces of $M\times\{t\}$,
and whose height trajectories  are asymptotic lines.
Vertical catenoids, in turn, have nonzero constant mean curvature hypersurfaces
as horizontal sections, and lines of curvature
as height trajectories.

In this setting, we show that vertical helicoids of $M\times\R$  have
all the classical uniqueness properties of the standard helicoids of $\R^3$. Namely, they
are minimal hypersurfaces of $M\times\R$ and, as such, they are the only ones which
are foliated by horizontal minimal hypersurfaces. They are also the only minimal
local graphs of harmonic functions (defined on domains in $M$), and the only
minimal non totally geodesic hypersurfaces of $M\times\R$ whose spacelike pieces are maximal with respect
to the standard Lorentzian product metric of $M\times\R$.

This last property extends
the  analogous classical result, set in Lorentzian space $\mathbb{L}^3,$
established by O. Kobayashi \cite{kobayashi}.
In our approach, we briefly  consider the class of hypersurfaces of $M\times\R$ whose mean
curvatures with respect to  both the Riemannian and Lorentzian metrics of $M\times\R$ coincide.
We call them \emph{mean isocurved}.
These hypersurfaces have been studied by Albujer-Caballero \cite{albujer1} in the case
where the ambient space is  $\mathbb{L}^{3}$  (see \cite{albujer2} as well).
Actually, during the preparation of this paper, we became acquainted with the recent works
by Alarcón-Alias-Santos \cite{alarcon-alias-santos} and Albujer-Caballero \cite{albujer-caballero}
which have some overlapping with ours on this subject.
Mean isocurved surfaces in $\h^2\times\R$ and $\s^2\times\R$ have also been considered
by Kim et al in \cite{kimetal}.

Concerning examples of  vertical helicoids in $M\times\R$,
we show that they can be constructed
by considering one-parameter groups of isometries of $M$ acting on
suitable  minimal hypersurfaces. 
When $M$ is one of the simply connected space forms,
this method  allows us to construct properly
embedded minimal vertical helicoids in $M\times\R$
which are foliated by vertical translations of totally geodesic hypersurfaces of $M.$
We also construct properly embedded vertical helicoids
in $\R^n\times\R$ and $\h^3\times\R$ which are foliated by vertical translations of
helicoids of $\R^{n}$ and
$\h^3,$ respectively. In the same way, we construct  vertical helicoids
in $\s^3_\delta\times\R,$ where $\s^3_\delta$ is a Berger sphere.
Finally, we  obtain a family of properly embedded
minimal vertical helicoids in $\s^{2n+1}\times\R$ which are foliated by $2n$-dimensional Clifford
tori, and also a corresponding family of
vertical helicoids in $\R^{2n+2}\times\R$ (previously constructed
by Choe and Hoppe \cite{choe-hoppe}), whose horizontal sections
are the cones of these tori in $\R^{2n+2}.$

Other examples of vertical helicoids that we
give are graphs of harmonic and horizontally homothetic
functions defined on domains of certain manifolds $M,$ such as the Nil and Sol
3-dimensional spaces (see Section \ref{sec-graphs}). We remark that, all the vertical helicoids
presented here, graphs or not, contain  spacelike zero mean isocurved open sets.

We also give a local characterization of vertical helicoids of $M\times\R$
with totally geodesic horizontal sections and  nonvanishing angle function  by showing that
each of its points has a neighborhood which can be expressed as  a ``twisting'' of a totally geodesic
hypersurface of $M$ (see Section \ref{subsec-constructionhelicoids} for more details).

Regarding vertical catenoids in $M\times\R$, their study  naturally leads to the consideration of
a broader class of hypersurfaces of $M\times\R;$ those which have the gradient of their height functions as
a principal direction. These hypersurfaces have been given a local characterization by R. Tojeiro \cite{tojeiro}
assuming that $M$ is one of the simply connected space forms. Here, we extend this result to general products
$M\times\R$ and conclude that a necessary and sufficient condition for the existence of minimal
or constant mean curvature (CMC) hypersurfaces in $M\times\R$ with this property (in particular,
vertical catenoids) is that $M$ admits families of isoparametric hypersurfaces.

This extension of Tojeiro's result, in fact, provides a way  of constructing such minimal
and CMC hypersurfaces (as long as they are admissible) by solving a first order linear differential equation.
This can be performed, for instance, when $M$ is any of the simply connected space forms, a Damek-Ricci
space or any of the simply connected 3-homogeneous manifolds with isometry group of dimension 4:
$\mathbb{E}(k,\uptau), \, k-4\uptau^2\ne 0.$ This result will also be applied for constructing properly embedded
vertical catenoids in $M\times\R$ when $M$ is a Hadamard manifold or the sphere $\s^n.$
As a further application, we give a complete classification  of hypersurfaces of $M\times\R$ whose angle function is constant.

The paper is organized as follows. In Section \ref{sec-preliminaries},
we set some notation and formulae. In Section \ref{sec-basiclemmas},
we introduce mean isocurved hypersurfaces and  establish some basic lemmas. We discuss on vertical helicoids
in Section \ref{sec-helicoids}. In Section \ref{sec-canonicaldirection}, we consider
hypersurfaces of $M\times\R$ which have the gradient of their  height functions as a principal direction.
Finally, in  Section \ref{sec-catenoids}, we discuss on vertical catenoids.

\section{Preliminaries} \label{sec-preliminaries}
Throughout this paper,  $M$ will denote  an arbitrary $n(\ge 2)$-dimensional $C^\infty$ Riemannian manifold.
For such an $M,$ we will consider the product manifold $M\times\R$ with its standard differentiable
structure. We will  set  $T(M\times\R)=TM\oplus T\R$ for the tangent bundle of $M\times\R,$
where $TM$ and $T\R$ stand for the tangent bundles of $M$ and $\R,$ respectively.
We will  endow $M\times\R$ with  the  Riemannian product metric:
\[
\langle \,,\,\rangle= \langle \,,\,\rangle_{\scriptscriptstyle M}+dt^2.
\]

We shall write $\pi_{\scriptscriptstyle M}$ and $\pi_{\scriptscriptstyle\R}$ for the projection of $M\times\R$
on its first and second factors, respectively, and  $\partial_t$ for the gradient
of $\pi_{\scriptscriptstyle\R}$ with respect to the Riemannian metric $\langle \,,\,\rangle.$
We remark that $\partial_t$ is a parallel field on $M\times\R.$

Let $\Sigma$ be an orientable  hypersurface  of  $M\times\R.$
Given a unit normal field $N\in T\Sigma^\perp\subset TM,$
we will denote by  $A$ the shape operator of $\Sigma$ relative to $N,$ that is,
\[
AX=-\overbar\nabla_XN, 
\]
where $\overbar\nabla$ stands for the Levi-Civita connection of $M\times\R$.
The gradient of a differentiable function $\zeta$ on $\Sigma$ will be denoted by $\nabla\zeta.$

The \emph{height function} $\xi$ and the \emph{angle function} $\theta$ of $\Sigma$
are defined as
\[\xi:=\pi_{\scriptscriptstyle\R}|_{\Sigma} \quad\text{and}\quad  \theta:=\langle N,\partial _t\rangle.\]
Regarding these functions, the following fundamental identities hold:
\begin{equation}  \label{eq-defgradxi}
\nabla\xi=\partial_t-\theta N\quad\text{and}\quad \nabla\theta=-A\nabla\xi,
\end{equation}
where the second one follows  from the fact that
$\partial_t$ is parallel in $M\times\R.$
We point out that $\theta_{\ssr}\in [-1,1],$ and 
that $x\in\Sigma$ is a critical point of the height function $\xi$ if and only if $\theta^2(x)=1.$ If so, we say
that $x$ is a \emph{horizontal point} of \,$\Sigma.$ Any field $X\in TM\subset T(M\times\R)$ will  be
called \emph{horizontal} as well.

\section{Basic Lemmas}  \label{sec-basiclemmas}


Given a product manifold $M\times\R,$  for each $t\in\R,$ we will call
the submanifold $M_t:=M\times\{t\}$ a \emph{horizontal section} of $M\times\R.$
If $\Sigma$ intersects a horizontal section $M_t$ transversally, we call
the set
\[\Sigma_t:=M_t\cap\Sigma\]
a \emph{horizontal section} of the hypersurface $\Sigma.$

Notice that, for all $t\in\R,$ $M_t$ is isometric to $M,$
and that any horizontal section $\Sigma_t$ is a hypersurface of $M_t$\,.
In this setting, it is easily checked that
\begin{equation}\label{eq-horizontalnormal}
\eta:=\phi_{\ssr}(N_{\ssr}-\theta_{\ssr}\partial_t), \,\,\,\phi=-(1-\theta^2)^{-1/2},
\end{equation}
is a well defined  unit normal field to $\Sigma_t$\,.

Now, denote the shape operator of \,$\Sigma_t$ with respect to $\eta$  by $A_\eta$\,, and set
$H$ and $H_{\Sigma_t}$ for the  (non normalized) mean curvature functions of $\Sigma$ and $\Sigma_t$, respectively.

\begin{lemma} \label{lem-horizontalsection}
Let $\Sigma_t$ be a horizontal section of a  hypersurface $\Sigma$ of $M\times\R.$
Then
\[
\langle A_\eta X,Y\rangle=\phi\langle AX,Y\rangle \,\,\, \forall X, Y\in T\Sigma_t\,.
\]
As a consequence,  for $T=\nabla\xi/\|\nabla\xi\|$, the following equality holds along $\Sigma_t$:
\begin{equation} \label{eq-meancurvatureslices}
H_{\Sigma_{t}}=\phi_{\ssr}(H_{\ssr}-\langle AT,T\rangle).
\end{equation}
\end{lemma}

\begin{proof}
We have that $M_t=M\times\{t\}$ is  totally geodesic in $M\times\R.$ Hence, its Riemannian connection
coincides with the restriction of the Riemannian connection $\overbar\nabla$ of $M\times\R$ to $TM_t\times TM_t$\,.
Therefore, for all $X\in T\Sigma_t$\,, we have
\[
A_\eta X=-\overbar\nabla_X\eta=-\overbar\nabla_X\phi_{\ssr}(N_{\ssr}-\theta_{\ssr}\partial_t)=-X(\phi_{\ssr})(N_{\ssr}-\theta_{\ssr}\partial_t)
+\phi_{\ssr}(A_{\ssr}X+X(\theta_{\ssr})\partial_t).
\]

Thus, for all $Y\in T\Sigma_t=TM_t\cap T\Sigma$\,,
\[
\langle A_\eta X,Y\rangle=\phi\langle AX,Y\rangle.
\]

Now,  in a suitable neighborhood $U\subset\Sigma$ of an arbitrary  point on $\Sigma_t$\,, consider an orthonormal frame
$\{X_1\,, \dots , X_{n-1},T\}$ such that $X_1\,, \dots X_{n-1}$ are all tangent to $\Sigma_t$\,.
Then, on $U\cap\Sigma_t$\,, we have
\[
H_{\Sigma_{t}}=\sum_{i=1}^{n-1}\langle A_\eta X_i,X_i\rangle=\phi\sum_{i=1}^{n-1}\langle A X_i,X_i\rangle=
\phi(H-\langle AT,T\rangle),
\]
which concludes the proof.
\end{proof}

\subsection{Mean Isocurved Hypersurfaces} 
Let us consider in $M\times\R$ the \emph{Lorentzian} product metric,
which is defined as
\[\langle \,,\,\rangle_{\ssl}:= \langle \,,\,\rangle_{\scriptscriptstyle M}-dt^2.\]
This metric relates to the Riemannian metric $\langle \,,\, \rangle$ of $M\times\R$ through the identity
\begin{equation} \label{eq-lorentzianeuclidean}
\langle X,Y\rangle_{\ssl} = \langle X,Y\rangle-2\langle X,\partial _t\rangle_{\ssr}\langle Y,\partial_t\rangle_{\ssr},
\end{equation}
which, as one can  verify, is valid for all $X,Y\in T(M\times\R).$

Denote by $\Sigma_{\ssl}:=(\Sigma,\langle \,,\, \rangle_{\ssl})$
a hypersurface $\Sigma$ of $M\times\R$
with the  induced Lorentzian metric of $M\times\R.$
We say that $\Sigma$ is \emph{spacelike} if $\Sigma_{\ssl}$ is a Riemannian manifold,
that is, the Lorentzian metric on $\Sigma$ is positive definite. It is easily checked that
$\Sigma$ is spacelike if and only if
$\langle Z,Z\rangle_{\ssl}<0$ for all nonzero local field $Z\in T\Sigma_{\ssl}^\perp.$
Also, any spacelike hypersurface of $M\times\R$ is necessarily orientable.

Assuming $\Sigma\subset M\times\R$ spacelike, choose a unit normal
$N_{\ssl}$ to $\Sigma_{\ssl}$, that is,
\[
\langle N_{\ssl}, N_{\ssl}\rangle_{\ssl}=-1 \quad\text{and}\quad \langle X, N_{\ssl}\rangle_{\ssl}=0 \,\, \forall X\in T\Sigma.
\]
It is a well known fact that the connections of $M\times\R$
with respect to the Riemannian and Lorentzian metrics coincide. So,
keeping the notation of Section \ref{sec-preliminaries}, we define the Lorentzian shape operator
of $\Sigma_{\ssl}$ with respect to $N_{\ssl}$ as
\begin{equation}\label{eq-shapeoperators}
A_{\ssl}X:=-\overbar\nabla_XN_{\ssl}. 
\end{equation}

Finally, the (non normalized) Lorentzian  mean curvature $H_{\ssl}$ of
\,$\Sigma_{\ssl}$ is defined as
\[H_{\ssl}:=-{\rm trace}\,A_{\ssl}.\]

\begin{definition}
A spacelike hypersurface $\Sigma\subset M\times\R$ is said to be  \emph{mean isocurved} if its Riemannian
and Lorentzian mean curvature functions, $H$ and $H_{\ssl}$, coincide.  When $H=H_{\ssl}=0,$ we say that
$\Sigma$ is \emph{zero mean isocurved.}
\end{definition}

Let us consider the following map
\[
\Phi(X)=X-2\langle X,\partial _t\rangle_{\ssr}\partial_t,  \,\, X\in T(M\times\R),
\]
which is easily seen to be an involution, that is, $\Phi\circ\Phi$
is the identity map of $T(M\times\R).$ Moreover,
for all $X, Y\in T( M\times\R),$  the following identities hold:
\begin{equation}\label{eq-propertiesPsi}
\langle \Phi(X),Y\rangle=\langle X,Y\rangle_{\scriptscriptstyle L}\quad\text{and}\quad
    \langle \Phi(X),Y\rangle_{\scriptscriptstyle L}=\langle X,Y\rangle.
\end{equation}

Given an oriented  hypersurface $\Sigma\subset M\times\R$  with unit
normal $N,$ it follows from the second relation in  \eqref{eq-propertiesPsi} that
$\Phi(N)$ is a Lorentzian normal field on $\Sigma.$ Indeed,
\[
\langle \Phi(N),X\rangle_{\ssl}=\langle N,X\rangle=0 \,\,\, \forall X\in T\Sigma.
\]
Moreover, considering also the equality \eqref{eq-lorentzianeuclidean}, we have
\[
\langle \Phi(N_{\ssr}),\Phi(N_{\ssr})\rangle_{\ssl}=
\langle N_{\ssr},\Phi(N_{\ssr})\rangle_{\ssr}=
\langle N_{\ssr},N_{\ssr}\rangle_{\ssl}=1-2\Theta_{\ssr}^2,
\]
from which we conclude that $\Sigma$ \emph{is spacelike if and only if}
$2\theta^2>1.$ If so, set
\[
N_{\ssl}:=\mu\Phi(N), \,\,\,\,\, \mu:=\frac{-1}{\sqrt{2\Theta_{\ssr}^2-1}}<0,
\]
and write $A_{\ssl}$ for the shape operator of $\Sigma_{\ssl}$ with respect to $N_{\ssl}.$

\begin{lemma} \label{lem-shapeoperators}
Let $\Sigma$ be a spacelike hypersurface of $M\times\R$ with no horizontal points.
With the above notation,  the following identities hold:
\begin{itemize}[parsep=1ex]
  \item[\rm i)] $\langle A_{\ssl}X,Y\rangle_{\ssl}=\mu\langle A_{\ssr}X,Y\rangle_{\ssr} \,\,\, \forall X,Y\in T\Sigma.$
  \item[\rm ii)] $H_{\ssl}+\mu H_{\ssr}=\mu(1-\mu^2)\langle A_{\ssr}T,T\rangle_{\ssr}, \,\,\, T=\nabla\xi/\|\nabla\xi\|.$
\end{itemize}
\end{lemma}

\begin{proof}
Given \,$X,Y \in TM,$\, one has
\[
\langle A_{\ssl}X,Y\rangle_{\ssl}=
\langle \overbar\nabla_XY, N_{\ssl}\rangle_{\ssl}=
\langle \overbar\nabla_XY, \mu\Phi(N_{\ssr})\rangle_{\ssl}=
\mu\langle \overbar\nabla_XY, N_{\ssr}\rangle_{\ssr}=
\mu\langle A_{\ssr}X,Y\rangle_{\ssr},
\]
which proves (i).

Now, let us consider a point $x\in\Sigma$ and a basis
$\mathfrak B=\{X_1\,, \dots , X_n\}$ of $T_x\Sigma$  which is orthonormal with
respect to the  Riemannian metric $\langle \,,\, \rangle$.
Since $x$ is non horizontal, we can assume that $X_1\,, \dots ,X_{n-1}$ are horizontal, i.e.,
tangent to $M,$ and $X_n=T.$ Hence, by \eqref{eq-lorentzianeuclidean}, $\{X_1\,, \dots ,X_{n-1}\}$ is orthonormal
with respect to the Lorentzian metric $\langle\,,\,\rangle_{\ssl}$,  and
$\langle X_i, T\rangle_{\ssl}=0 \,\, \forall i=1,\dots ,n-1.$

Denote  by $[a_{ij}]$ and $[\ell_{ij}]$ the matrices of the shape operators
$A$\, and $A_{\ssl}$, respectively,  with respect to the basis $\mathfrak B.$
From (i), we have
\begin{equation}\label{eq-matrices1}
\ell_{ij}=\langle A_{\ssl}X_i, X_j \rangle_{\ssl}=
\mu\langle A_{\ssr}X_i,X_j\rangle_{\ssr}=\mu a_{ij} \,\,\, \forall i,j=1,\dots ,n-1.
\end{equation}
Also, for any index \,$j=1,\dots ,n,$ one has
\begin{equation}\label{eq-matrices0}
\mu a_{nj}=\mu\langle A_{\ssr}X_j,T\rangle_{\ssr}=
\langle A_{\ssl}X_j, T \rangle_{\ssl}=\sum_{i=1}^{n}\ell_{ij}\langle X_i,T\rangle_{\ssl}=\ell_{nj}\langle T, T \rangle_{\ssl}\,.
\end{equation}
However, by  \eqref{eq-defgradxi} and \eqref{eq-lorentzianeuclidean},
\[
\langle T, T \rangle_{\ssl}=1-2\langle T, \partial_t\rangle_{\ssr}^2=2\Theta_{\ssr}^2-1=\frac{1}{\mu^2}\,\cdot
\]
This, together with \eqref{eq-matrices0}, yields
\begin{equation}\label{eq-matrices2}
\ell_{nj}=\mu^3a_{nj}\quad\forall j=1,\dots ,n.
\end{equation}

Putting \eqref{eq-matrices1} and \eqref{eq-matrices2} together, we have
\[
[\ell_{ij}]=\mu
\left[
\begin{array}{ccc}
a_{11} & \cdots & a_{1n} \\
\vdots    &        & \vdots    \\
a_{i1} & \cdots & a_{in} \\
\vdots    &        & \vdots    \\
\mu^2a_{n1} & \cdots &\mu^2a_{nn}
\end{array}
\right],
\]
which implies that
\begin{equation}\label{eq-tracedet}
    {\rm trace}[\ell_{ij}]=\mu({\rm trace}[a_{ij}]+(\mu^2-1)a_{nn}).
\end{equation}

Since we have $a_{nn}=\langle AT,T\rangle,$
$H_{\ssl}=-{\rm trace}[\ell_{ij}],$ and $H={\rm trace}[a_{ij}],$ the identity
\eqref{eq-tracedet} clearly implies (iii).
\end{proof}

The following result extends \cite[Theorem 4]{albujer1}, set in Lorentzian space $\mathbb{L}^{3},$
to hypersurfaces in $M\times\R.$

\begin{corollary} \label{prop-noelipticpoints}
Let $\Sigma$ be a mean isocurved  hypersurface of $M\times\R.$
Then, its second fundamental form $\sigma$ is nowhere definite. Furthermore,
$\sigma$  is semi-definite
at $x\in\Sigma$ if and only if \,$\Sigma$  is totally geodesic at $x.$
\end{corollary}
\begin{proof}
Let us denote by $C\subset\Sigma$ the set of critical points of the height function
$\xi$ of $\Sigma.$
Keeping the notation of the proof of the preceding lemma, and considering
the  equality  \eqref{eq-tracedet}, we have that
$H=\mu(1-\mu)a_{nn}$ on $\Sigma-C,$
for $H_{\ssl}=H.$
Thus,
\begin{equation}\label{eq-trace}
\sum_{i=1}^{n-1}a_{ii}+(1+\mu(\mu-1))a_{nn}=0.
\end{equation}
However,  $1+\mu(\mu-1)>0$ and $a_{ii}=\langle AX_i,X_i\rangle_{\ssr}=\sigma_{\ssr}(X_i,X_i), \,i=1,\dots ,n.$
Hence, the equality \eqref{eq-trace} implies that, at a point $x$ in the closure of $\Sigma-C$ in $\Sigma,$
$\sigma_{\ssr}$ is  neither definite nor semi-definite,
unless, in the latter case, it vanishes.
\end{proof}

\section{Vertical Helicoids in $M\times\R.$} \label{sec-helicoids}

Inspired by some fundamental properties of the standard helicoids of $\R^3$
(see Example \ref{exam-2helicoids} below), we
introduce in this section the concept of  vertical helicoid in $M\times\R$. We shall  establish
the uniqueness properties of these hypersurfaces and present a variety of examples,
as we mentioned in the introduction. In addition, we will characterize the vertical helicoids
which are graphs of functions on $M,$ and give a local characterization of vertical helicoids $\Sigma$ whose
horizontal sections $\Sigma_t$ are totally geodesic in $M_t\,.$

\begin{definition} \label{def-helicoid}
Let $\Sigma$ be a hypersurface of $M\times\R$ with no horizontal points
and nonconstant angle function.   We say that $\Sigma$ is a \emph{vertical helicoid}
if it satisfies the following conditions:

\begin{itemize}[parsep=1ex]
  \item The horizontal sections $\Sigma_t\subset\Sigma$ are minimal hypersurfaces of $M\times\{t\}.$
  \item $\nabla\xi$ is an asymptotic direction of \,$\Sigma,$ that is, $\langle A\nabla\xi,\nabla\xi\rangle=0$ on \,$\Sigma.$
\end{itemize}
\end{definition}

\begin{remark}
Considering the standard helicoids in $\R^3=\R^2\times\R,$ one could expect that a
right extension of this concept to the context of products $M\times\R$ should ask for the horizontal sections
to be totally geodesic, since the horizontal sections of the helicoids in $\R^3$ are straight lines.
However, as our results and examples shall show, the appropriate condition to be imposed to  the horizontal sections is, in fact,
minimality, as in the above definition.
\end{remark}

\begin{remark} \label{rem-helices}
The identity $\nabla\theta=-A\nabla\xi$ implies that $\nabla\xi$ is an asymptotic direction
of $\Sigma$ if and only if the equality
$
\langle \g_{\ssr}\theta_{\ssr},\g_{\ssr}\xi\rangle_{\ssr}=0
$
holds on $\Sigma.$ In this case, we have that $\theta$ is constant along any trajectory $\gamma(s)$
of $\nabla\xi.$ However, $\langle\nabla\xi,\partial_t\rangle=1-\theta^2,$  which gives that the
tangent directions $\gamma'(s)$  make a constant angle with the vertical direction $\partial_t.$
Therefore, considering the concept of helix in $\R^3$ as a curve which makes a constant angle with a given direction,
we can extend it to curves in $M\times\R$ in an obvious way and conclude
that the trajectories of $\nabla\xi$ on a vertical helicoid in $M\times\R$ are \emph{vertical helices}.
\end{remark}

In what follows, let $Q_c^n$ denote the simply connected $n$-space form of constant sectional curvature $c\in\{0,1,-1\},$ that is,
the Euclidean space $\R^n$ ($c=0$), the $n$-sphere $\s^n$ ($c=1$), or the hyperbolic space $\h^n$ ($c=-1$).

\begin{example}[\emph{Helicoids in} $Q_c^2\times\R$] \label{exam-2helicoids}
Consider the following parametrization of the standard vertical helicoid $\Sigma$ of $\R^3=\R^2\times\R$
with \emph{pitch} $a>0,$
\[
\Psi(x,y)=(x\cos y,x\sin y,ay), \,\, (x,y)\in\R^2.
\]
As its Riemannian unit normal field, we can choose
\[
N=\frac{1}{\sqrt{x^2+a^2}}(a\sin y,-a\cos y,x),
\]
which gives $\theta=x/(x^2+a^2)^{1/2}.$

Since $\Psi$ is an orthogonal parametrization and
$\theta_{\ssr}$ depends only on $x,$ we have that $\g_{\ssr}\theta_{\ssr}$
is parallel to $\Psi_x=(\cos y,\sin y,0).$ In particular,
\[
\langle\g_{\ssr}\theta_{\ssr},\g_{\ssr}\xi\rangle_{\ssr}=\langle\g_{\ssr}\theta_{\ssr},\partial_t\rangle_{\ssr}=0.
\]
Hence, $\nabla\xi$ is an asymptotic direction of $\Sigma.$

We also have that all  horizontal sections of $\Sigma$ are  straight lines.
Therefore, $\Sigma$ is a  vertical helicoid as in Definition \ref{def-helicoid}.
Moreover, from the equality
\[2\theta_{\ssr}^2-1=\frac{x^2-a^2}{x^2+a^2}\,,\]
we conclude that the open subset $\Sigma'=\{\Psi(x,y)\in\Sigma\,;\, |x|>a\}$ is spacelike and, as is
well known, zero mean isocurved (see, e.g., \cite{kobayashi}).

Considering the standard inclusions  $\s^2\hookrightarrow\R^3$ and
$\h^2\hookrightarrow\mathbb{L}^3,$ we can apply an analogous reasoning to the
parametrizations (see, e.g., \cite[Section 4]{daniel}):
\begin{align}
  \Psi_{\rm sph}(x,y) &=(\cos x\cos y,\cos x\sin y, \sin x, ay)\in\s^2\times\R; \nonumber\\
  \Psi_{\rm hyp}(x,y) &=(\sinh x\cos y, \sinh x\sin y, \cosh x, ay)\in \h^2\times\R; \nonumber
\end{align}
and conclude that their images are  vertical helicoids in $\s^2\times\R$ and
$\h^2\times\R,$ respectively. They are both minimal surfaces
containing open  spacelike zero mean isocurved subsets, as verified in \cite{kimetal}.
\end{example}

We prove now, as suggested by the above examples, that
vertical helicoids in product spaces $M\times\R$ are minimal hypersurfaces. As such,
except for  some constant angle hypersurfaces,
they are the only ones foliated by horizontal minimal hypersurfaces.  Moreover,
spacelike pieces of vertical helicoids (if any) are  zero mean isocurved
hypersurfaces in $M\times\R$, and they are unique with respect to this property as well.

\begin{theorem} \label{th-HR=HL}
Let $\Sigma$ be a hypersurface of $M\times\R$ with no horizontal points
and nonconstant angle function. Then, the following statements are equivalent:
\begin{itemize}[parsep=1ex]
  \item[{\rm i)}] $\Sigma$ is a  vertical helicoid.
   \item[{\rm ii)}]$\Sigma_{\ssr}$ and all the horizontal sections $\Sigma_t$ are minimal hypersurfaces.
 \end{itemize}
If, in addition, $\Sigma$ is spacelike, then both {\rm (i)} and {\rm (ii)} are equivalent to:
\begin{itemize}
 \item[{\rm iii)}] $\Sigma$ is zero mean isocurved. 
\end{itemize}
\end{theorem}

\begin{proof}
(i) $\Rightarrow$ (ii): Since we are assuming that $\Sigma$ is a vertical helicoid,
we have $H_{\Sigma_t}=0$ for all horizontal sections  $\Sigma_t\subset\Sigma,$ and
$\langle A_{\ssr}\g_{\ssr}\xi\,,\g_{\ssr}\xi\rangle_{\ssr}=0$ on $\Sigma.$
Thus, from  the identity \eqref{eq-meancurvatureslices} in Lemma \ref{lem-horizontalsection},
$H_{\ssr}=0,$ that is, $\Sigma_{\ssr}$ is minimal.

\vt

\noindent
(ii) $\Rightarrow$ (i):
Now, we have $H_{\ssr}=H_{\Sigma_t}=0$ for any horizontal section $\Sigma_t\subset\Sigma.$ In this case,
\eqref{eq-meancurvatureslices} yields $\langle A_{\ssr}T,T\rangle_{\ssr}=0,$ which implies
that $\nabla\xi$ is an asymptotic direction, that is,
$\Sigma$ is a vertical helicoid.

\vt

\noindent
(ii) $\Rightarrow$ (iii): We have $H=0$ and, as above, $\langle A_{\ssr}T,T\rangle_{\ssr}=0.$
Hence, by Lemma \ref{lem-shapeoperators}-(iii), $H_{\ssl}=0,$ i.e.,
$\Sigma$ is zero mean isocurved.

\vt

\noindent
(iii) $\Rightarrow$ (ii): From
$H=H_{\ssl}=0$  and
Lemma \ref{lem-shapeoperators}-(iii), one has $\langle A_{\ssr}T,T\rangle_{\ssr}=0.$ This, together with
identity \eqref{eq-meancurvatureslices}, gives that
the horizontal sections $\Sigma_t\subset\Sigma$ are minimal hypersurfaces of $M\times\{t\}.$ Hence, $\Sigma$ is a
vertical helicoid.
\end{proof}

Vertical helicoids can be constructed  by ``twisting'' minimal hypersurfaces, as
shown in the following examples.

\begin{example}[\emph{Twisted planes in $\R^3\times\R$}] \label{exam-helicoidsRnxR}
Given $a, k>0,$ consider the map
\[
\Psi(x,y,s):=
\left[
\begin{array}{cccc}
  \cos ks & -\sin ks & 0 & 0 \\
  \sin ks & \phantom{-}\cos ks & 0 & 0 \\
  0 & \phantom{-}0 & 1 & 0 \\
  0 & \phantom{-}0 & 0 & 1
\end{array}
\right]
\left[
\begin{array}{c}
   x \\
   0 \\
   y \\
   0
\end{array}
\right]+
a\left[
\begin{array}{c}
   0 \\
   0 \\
   0\\
   s
\end{array}
\right], \,\, (x,y,s)\in\R^3,
\]
which  we call a \emph{vertical twisting}
of the plane $\R^2\times\{0\}\subset\R^3$ in $\R^3\times\R.$
It is easily verified that
$\Psi$ is a parametrization of a properly  embedded hypersurface $\Sigma$ of $\R^3\times\R$.
Also, direct computations give that 
\[
N=\frac{(a\sin s,-a\cos s, 0,kx)}{\sqrt{a^2+(kx)^2}}
\]
is a unit normal field on $\Sigma$.
In particular, $\theta=kx/\sqrt{a^2+(kx)^2}$ depends only on $x$ and
$\theta^2\ne 1,$ that is, $\xi$ has no critical points on $\Sigma.$
Also, the inverse matrix
$[g^{ij}]$ of the first fundamental form of $\Sigma$ in this parametrization is
\[
[g^{ij}]=\left[
\begin{array}{ccc}
  1 & 0 & 0 \\
  0 & 1 & 0 \\
  0 & 0 & \frac{1}{a^2+x^2}
\end{array}\right].
\]
Therefore,
\[
\nabla\theta=\frac{\partial\theta}{\partial x}\frac{\partial\Psi}{\partial x}=
\frac{\partial\theta}{\partial x}(\cos s,\sin s, 0, 0) \,\Rightarrow
\, \langle\nabla\theta,\nabla\xi\rangle=\langle\nabla\theta,\partial_t\rangle=0.
\]
Thus, $\Sigma$ is a (minimal) vertical helicoid, since its horizontal sections $\Sigma_t$ are planes of $\R^3\times\{t\}.$
Moreover, its angle function $\theta$ satisfies
\[
2\theta^2-1=\frac{(kx)^2-a^2}{(kx)^2+a^2}\,,
\]
which implies that the nonempty open subset $\Sigma'$ of $\Sigma$ given by
\[
\Sigma':=\{\Psi(x,y,s)\in\Sigma\,;\, |x|>a/k\}
\]
is spacelike. So, by Theorem \ref{th-HR=HL}, $\Sigma'$ is
zero mean isocurved in $\R^3\times\R.$
\end{example}

\begin{example}[\emph{Twisted helicoids in $\R^n\times\R$}] \label{exam-twistedhelicoid}
Let us consider now the map
\[
\Psi(x,y,s):=
\left[
\begin{array}{cccc}
  \cos ks & -\sin ks & 0 & 0 \\
  \sin ks & \phantom{-}\cos ks & 0 & 0 \\
  0 & \phantom{-}0 & 1 & 0 \\
  0 & \phantom{-}0 & 0 & 1
\end{array}
\right]
\left[
\begin{array}{c}
   x\cos y \\
   x\sin y \\
   y \\
   0
\end{array}
\right]+
a\left[
\begin{array}{c}
   0 \\
   0 \\
   0\\
   s
\end{array}
\right], \,\, (x,y,s)\in\R^3.
\]

Clearly $\Sigma=\Psi(\R^3)$ is a  properly embedded hypersurface of $\R^3\times\R,$ which we call
a \emph{twisted helicoid}. A unit normal field to $\Sigma$ is
\[
N=\frac{1}{\sqrt{a^2(1+x^2)+(kx)^2}}(a\sin(y+ks),-a\cos(y+ks), ax,kx),
\]
so that $\theta=kx/\sqrt{a^2(1+x^2)+(kx)^2}.$ Again, we have $\theta^2\ne 1$ and
\[
\nabla\theta=g^{11}\frac{\partial\theta}{\partial x}\frac{\partial\Psi}{\partial x}=
g^{11}\frac{\partial\theta}{\partial x}(\cos(y+ks),\sin(y+ks),0,0),
\]
which yields $\langle \nabla\theta,\nabla\xi\rangle=0.$

Since, by construction, the
horizontal sections $\Sigma_t$ of $\Sigma$ are two-dimensional helicoids in $\R^3\times\{t\},$ we conclude
from the above that $\Sigma$ is a vertical helicoid in $\R^3\times\R.$ Moreover, if $k>a,$ then the set
\[
\Sigma':=\{\Psi(x,y,s)\in\Sigma\,;\, |x|>a/\sqrt{k^2-a^2} \}
\]
is easily seen to be spacelike and, so, zero mean isocurved.

Now, define the functions $f,\, g\colon\R^n\rightarrow\R$ by
\begin{align}
f(x_2\,,\dots ,x_{n-1},s) &=\cos(x_2+x_3+\cdots +x_{n-1}+s). \nonumber\\
g(x_2\,,\dots ,x_{n-1},s) &=\sin(x_2+x_3+\cdots +x_{n-1}+s). \nonumber
\end{align}
Applying induction on $n$ and proceeding as above, one concludes that the map
\[
\Psi(x_1\,, \dots ,x_{n-1},s)= (x_1f(x_2\,,\dots ,ks),x_1g(x_2\,,\dots ,ks),x_2\,,x_3\,, \dots ,x_{n-1},as)
\]
parametrizes a properly embedded minimal vertical helicoid $\Sigma^n\subset\R^n\times\R$ whose horizontal sections are
vertical helicoids in $\R^{n-1}\times\R.$ Furthermore, for $k>a,$ $\Sigma$
contains open spacelike zero mean isocurved subsets.
\end{example}

\begin{example}[\emph{Twisted Clifford torus in \,$\s^3\times\R$}] \label{exam-helicoidsS3xR}
Given $k>0,$ consider the immersion
\[
\Psi\colon\R^3\rightarrow\s^3\times\R\subset\R^5
\]
defined by the equality
\[\Psi(x,y,s)=(\cos \left( x+ks \right) \cos y ,\sin\left( x+ks
 \right) \cos y  ,\cos  x  \sin  y
  ,\sin  x  \sin y ,s).
\]
Then, $\Sigma=\Psi(\R^3)$ is proper and embedded in $\s^3\times\R.$
A computation shows that
\[
N=\frac{( {\sin y  \sin \left( x+ks \right)},-{\sin y\cos
 \left( x+ks \right)},- {\sin x\cos y}, {\cos x\cos y},{k\cos y\sin y})}{\sqrt{1+(k\cos y \sin y)^2}}
\]
is a unit normal to $\Sigma,$ which implies that  its angle function is given by
\[
\theta=\frac{k\cos y\sin y}{\sqrt{1+(k\cos y\sin y)^2}}=\frac{k\sin(2y)/2}{\sqrt{1+k^2\sin^2(2y)/4}}\,\cdot
\]
Also, the matrix $[g_{ij}]$ of the first fundamental form of $\Sigma$ is
\[
[g_{ij}]=\left[
\begin{array}{ccc}
  1 & 0 & k\cos^2y \\
  0 & 1 & 0 \\
  k\cos^2y & 0 & k^2\cos^2y+1
\end{array}
\right].
\]
In particular,  for its inverse $[g^{ij}],$ we have that
$g^{12}=g^{32}=0,$ since the corresponding cofactors of $[g_{ij}]$ clearly vanish.
This, together with the fact that $\theta$ depends only on $y,$ gives that
\[
\nabla\theta=g^{22}\frac{\partial\theta}{\partial y}\frac{\partial\Psi}{\partial y}
\,\Rightarrow\, \langle\nabla\theta,\nabla\xi\rangle =\langle\nabla\theta,\partial_t\rangle = 0,
\]
for  ${\partial\Psi}/{\partial y}$ is a horizontal vector.
Therefore, $\nabla\xi$ is an asymptotic direction of $\Sigma.$
Observing that each horizontal section of $\Sigma$ is a
Clifford torus, which is a compact embedded   minimal hypersurface of $\s^3,$  we conclude that
$\Sigma$ is a properly embedded minimal vertical helicoid of $\s^3\times\R.$

Finally, we have that the angle function of $\Sigma$ satisfies
\[
2\theta^2-1= 
\frac{k^2\sin^2(2y)-4}{k^2\sin^2(2y)+4}\,\cdot
\]
Hence, if we assume  $k>2,$ we have that the open set
\[
\Sigma':=\{\Psi(x,y,s)\in\Sigma\,;\, y>\arcsin(2/k)/2\}\subset\Sigma
\]
is nonempty and zero mean isocurved in $\s^3\times\R.$
\end{example}

\begin{example}[\emph{Twisted hyperbolic helicoid in $\h^3\times\R$}] \label{exam-twistedhelicoid}
Consider the Lorentzian model of hyperbolic space
$\h^3\hookrightarrow\mathbb{L}^4=(\R^4, ds^2), \,\, ds^2=dx_1^2+dx_2^2+dx_3^2-dx_4^2.$
It is well known that the map
\[
(x,y)\in\R^2\mapsto (\sinh x\cos y,\sinh x\sin y, \cosh x\sinh y,\cosh x\cosh y)\in\h^3
\]
parametrizes a properly embedded minimal surface which is called the \emph{hyperbolic helicoid}
of $\h^3.$
Considering its twisting $\Psi:\R^3\rightarrow\h^3\times\R$ defined, for $k>0,$  by
\[
\Psi(x,y,s)=(\sinh x\cos (y+ks),\sinh x\sin (y+ks), \cosh x\sinh y,\cosh x\cosh y,as),
\]
we have that the hypersurface $\Sigma=\Psi(\R^3)$ is proper and embedded in $\h^3\times\R.$
A unit normal field for $\Sigma$ is given by
\[
N=\lambda
\left[
\begin{array}{c}
\phantom{-}\cosh x\sin(y+ks)\\
-\cosh x\cos(y+ks)\\
\phantom{k}\sinh x\cosh y\\
\phantom{k}\sinh x\sinh y\\
k\sinh x\cosh x
\end{array}
\right],
\]
where $\lambda=(\cosh^2x+\sinh^2x+(k\cosh x\sinh x)^2)^{-1/2}$.
Therefore, the angle function  of $\Sigma$
is $\theta=k\lambda\sinh x\cosh x,$ which depends only on $x.$

Proceeding as before, one easily
concludes that $\nabla\theta$ is horizontal, i.e., that $\nabla\xi$ is an asymptotic direction of $\Sigma.$
Hence, $\Sigma$ is a properly embedded minimal vertical helicoid in $\h^3\times\R$ whose
horizontal sections $\Sigma_t$ are hyperbolic helicoids of $\h^3\times\{t\}.$ Also, for
sufficiently large $k$,  $\Sigma$ contains open spacelike zero mean isocurved subsets.
\end{example}

\begin{example}[\emph{Twisted helicoid in $\s_\delta^3\times\R$}] \label{exam-berger}
Consider the product $\mathbb{S}_{\delta}^3\times \mathbb{R}$, where the first factor is a Berger sphere.
It is  well known  that, given $\alpha\in\R,$ the map
\[
(s,\uptau)\in\R^2\mapsto (e^{i\alpha s}\cos(\uptau),e^{is}\sin(\uptau))\in \mathbb{S}_{\delta}^3
\]
is a parametrization of a minimal helicoid of $\mathbb{S}_{\delta}^3$ (see, for instance, \cite{shimetal}).

From this helicoid, using the same twisting method of the previous examples,
we obtain a vertical helicoid in $\mathbb{S}_{\delta}^3\times \mathbb{R}$
that is given by
\[
\Psi(s,\uptau,u)=(e^{i(\alpha s+u)}\cos(\uptau),e^{i(s+u)}\sin(\uptau),au), \,\,\, a\ne 0.
\]

To see that $\Psi$ is indeed a vertical helicoid, it suffices to compute the angle
function $\theta$ and check that its gradient is horizontal.
After a long but straightforward computation, $\theta$ can be written as
\[
\theta=\frac{-\alpha\cos(\uptau)\sin(\uptau)}{\omega(\uptau)}\,,
\]
where $\omega(\uptau)$ is given by
\begin{eqnarray}
\omega(\uptau) &= & [\cos^4(\uptau)((1-\delta^2)\delta^2(\alpha+1)^2a^2-\alpha^2)\nonumber \\
             &+ & \cos^2(\uptau)(\delta^2(\alpha+1)(\delta^2(\alpha+1)-2)a^2+\alpha^2)+\delta^2a^2]^{1/2}.\nonumber
\end{eqnarray}

From these expressions, and after some further computations,
we get  that $\nabla\theta$ is horizontal. Also, for a convenient
choice of the parameters $\alpha,\,a, \delta$,  and of the range of $s,\,\uptau,$ and $u$,
$\Psi$ is a spacelike immersion.
\end{example}

\subsection{Vertical Helicoids as Graphs} \label{sec-graphs}

Let $u$ be a differentiable (i.e., $C^\infty$) function defined on a domain $\Omega\subset M.$
It is easily checked that
\begin{equation}\label{eq-normaltograph}
N=\frac{-\nabla u+\partial_t}{\sqrt{1+\|\nabla u\|^2}}\,,
\end{equation}
is a unit normal to $\Sigma={\rm graph}(u)\subset M\times\R$,
where,  by abuse of notation, we are writing $\nabla u$ instead of $\nabla u\circ\pi_{\scriptscriptstyle M}.$
In particular,
\begin{equation}\label{eq-thetaandu}
\theta=\frac{1}{\sqrt{1+\|\nabla u\|^2}} 
\end{equation}
is the angle function of $\Sigma.$

Denoting  by ${\rm div}$ the divergence of fields on $M,$ as is well known,
$\Sigma={\rm graph}(u)$ is a minimal hypersurface of $M\times\R$ if and only if  $u$  satisfies the equation
\begin{equation}\label{eq-meancurvature0}
{\rm div}\left(\frac{\nabla u}{\sqrt{1+\|\nabla u\|^2}}\right)=0.
\end{equation}

\begin{lemma} \label{lem-graph}
Let $\Sigma$ be the graph of a differentiable function $u$ on a domain $\Omega\subset M$, and
let \,$\Sigma_t$ be a horizontal section of \,$\Sigma.$ Then, the following hold:
\begin{itemize}[parsep=.5ex]
  \item[{\rm i)}] $\Sigma$ is minimal in $M\times\R$ if and only if $u$ satisfies:
  \begin{equation}\label{eq-meancurvature}
   \Delta u-\frac{\|\nabla u\|}{1+\|\nabla u\|^2}\langle \nabla u, \nabla\|\nabla u\|\rangle=0.
\end{equation}
  \item[{\rm ii)}] The mean curvature  of \,$\Sigma_t$ is given by:
 \begin{equation}\label{eq-Hhorizontal}
H_{\Sigma_t}=\frac{\Delta u}{\|\nabla u\|}-\frac{\langle\nabla u,\nabla\|\nabla u\|\rangle}{\|\nabla u\|^2}\,\cdot
\end{equation}
\end{itemize}
\end{lemma}

\begin{proof}
Given a differentiable function $\rho$ on $\Omega,$ it is an elementary fact that
\begin{equation}\label{eq-elementary}
{\rm div}(\rho\nabla u)=\rho\Delta u+\langle\nabla\rho,\nabla u\rangle.
\end{equation}
Then, considering   \eqref{eq-meancurvature0} and setting
$\rho=1/\sqrt{1+\|\nabla u\|^2}$, one  easily concludes that
the equations \eqref{eq-meancurvature0} and  \eqref{eq-meancurvature} are equivalent.

From \eqref{eq-normaltograph}, we have that
$\eta=-\nabla u/\|\nabla u\|$
is a unit normal field to $\Sigma_t$\,.
Therefore, if we choose  an   orthonormal frame
$\{X_1\,,\dots ,X_{n-1}\}$ in $T\Sigma_t$\,, we have
\[
H_{\Sigma_t}=\sum_{i=1}^{n-1}-\left\langle\overbar\nabla_{X_i}\eta,X_i\right\rangle={\rm div}\,\frac{\nabla u}{\|\nabla u\|}\,\cdot
\]
Now, equality \eqref{eq-Hhorizontal} follows from \eqref{eq-elementary} if we set  $\rho=1/\|\nabla u\|.$
\end{proof}

The identities in the above lemma suggest  the consideration  of
horizontally homothetic functions, which we now introduce (cf. \cite{ou1, ou2}).

\begin{definition}
We say that a smooth function $u$
on $\Omega\subset M$ is \emph{horizontally homothetic} if the identity
$\langle\nabla u,\nabla\|\nabla u\|\rangle=0$ holds on $\Omega.$
\end{definition}

Our next result establishes the uniqueness of vertical helicoids
as  minimal hypersurfaces which are local graphs of harmonic functions.

\begin{theorem} \label{th-isocurvedgraph}
Let $\Sigma={\rm graph}(u),$ where $u$ is a
smooth function defined on a domain $\Omega\subset M$ whose gradient never vanishes. Then,
if the angle function of $\Sigma$ is nonconstant, the following are equivalent:
\begin{itemize}[parsep=1ex]
  \item[{\rm i)}]  $\Sigma$ is a vertical helicoid in $M\times\R.$
  \item[{\rm ii)}] $u$ is harmonic and $\Sigma$ is minimal.
  \item[{\rm iii)}] $u$ is harmonic and horizontally homothetic.
\end{itemize}
\end{theorem}

\begin{proof}
Assume that $\Sigma$ is a vertical helicoid. Then, $H_{\Sigma_t}=0$ for any
horizontal section $\Sigma_t$ of $\Sigma.$ Also,
by Theorem \ref{th-HR=HL}, $\Sigma$ is minimal. So, by Lemma \ref{lem-graph},
$u$ satisfies  equation \eqref{eq-meancurvature}.
Combining it with \eqref{eq-Hhorizontal}, we have
\[
\frac{\langle\nabla u,\nabla\|\nabla u\|\rangle}{\|\nabla u\|(1+\|\nabla u\|^2)}=0,
\]
which yields $\langle\nabla u,\nabla\|\nabla u\|\rangle=0.$
This, together with \eqref{eq-meancurvature}, implies that $u$ is a harmonic
function, that is, (i) $\Rightarrow$ (ii).

Let us suppose now that (ii) holds. Then,  $u$ satisfies \eqref{eq-meancurvature}. Since $u$ is harmonic,
it follows that $u$ is also  horizontally homothetic. Now, we have from \eqref{eq-Hhorizontal} that the horizontal sections of
$\Sigma$ are minimal. Hence, from Theorem \ref{th-HR=HL},  $\Sigma$ is a vertical helicoid, which shows that
(i) and (ii) are equivalent.

The equivalence between (ii) and (iii) follows directly from Lemma \ref{lem-graph}-(i).
\end{proof}

We now make use of Theorem \ref{th-isocurvedgraph} to obtain vertical helicoids $\Sigma\subset M\times\R$ which contain
spacelike pieces of zero mean isocurved hypersurfaces. Before that, let us remark that, by \eqref{eq-thetaandu},
the angle function $\theta$ of $\Sigma={\rm graph}(u)$ satisfies
\[
2\theta^2-1=\frac{(1-\|\nabla u\|^2)}{(1+\|\nabla u\|^2)}\,\cdot
\]
Therefore, $\Sigma={\rm graph}(u)$ is a spacelike hypersurface if and only if $\|\nabla u\|<1.$

\begin{example}
Consider the set  $\Omega$ of points $(x_1\,, \dots ,x_n)\in\R^n$  which satisfy $x_{n-1}>0$  and define
on it the function
\[
u(x_1\,, \dots ,x_n)=\sum_{i=1}^{n-2}a_ix_i+b\arctan(x_n/x_{n-1}).
\]
From a direct computation, one concludes that $u$ is harmonic and horizontally homothetic.
Thus, Theorem \ref{th-isocurvedgraph} gives that $\Sigma={\rm graph}(u)$ is a vertical helicoid.
Moreover, the gradient of $u$ is
\[
\nabla u(x_1\,,\dots ,x_n)=\left(a_1\,, \dots ,a_{n-2}\,, \frac{-bx_n}{x_{n-1}^2+x_n^2}\,, \frac{bx_{n-1}}{x_{n-1}^2+x_n^2}\right),
\]
which implies that
\begin{equation}\label{eq-normgrad}
\|\nabla u\|^2=\sum_{i=1}^{n-2}a_i^2+\frac{b^2}{x_{n-1}^2+x_n^2}\,\cdot
\end{equation}

Therefore, if we assume  $a_1^2+\cdots +a_{n-2}^2<1$ and consider the set $\Omega'\subset\Omega$
of points $(x_1\,,\dots ,x_n)\in\Omega$ for which the right hand side of \eqref{eq-normgrad} is $<1,$ we have that
$\Sigma'={\rm graph}(u|_{\Omega'})$ is spacelike and, in particular, zero mean isocurved in $\R^n\times\R.$
\end{example}

\begin{example}[\emph{Y-L Ou examples}] The following functions
$u:M\rightarrow\R$, which were considered by
Y-L Ou in \cite{ou1,ou2}, are all harmonic and horizontally homothetic.
Therefore, by Theorem \ref{th-isocurvedgraph}, their graphs are complete embedded
vertical helicoids in the corresponding product $M\times\R.$
\vt

\begin{itemize}[parsep=1.2ex]
\item[i)] $M=\h^n=(\R^n_+,x_n^{-2}g_{\scriptscriptstyle \rm Euc}),$\,\, $u(x_1,\dots ,x_n)=ax_i\,, \,\,\,  1\le i\le n-1.$

\item[ii)] $M=(\R^3,g_{\scriptscriptstyle \rm Nil}),$  \,$g_{\scriptscriptstyle\rm Nil}=dx^2+dy^2+(dz-xdy)^2,$\,\,\, $u(x,y,z)=a(z-xy/2).$

\item[iii)] $M=(\R^3,g_{\scriptscriptstyle \rm Sol}),$ \,$g_{\scriptscriptstyle\rm Sol}=e^{2z}dx^2+e^{-2z}dy^2+dz^2,$ \,\,\, $u(x,y,z)=az.$
\end{itemize}
\vt

We remark that, in contrast with (i), in (ii) and (iii) the horizontal sections of $\Sigma={\rm graph}(u)$ are non totally geodesic.
Also, in all cases, for certain suitable values of the parameter $a,$ $\Sigma$ has nonempty spacelike zero mean isocurved open sets.
\end{example}

\subsection{Construction and Local Characterization of Vertical Helicoids} \label{subsec-constructionhelicoids}
In this section, we generalize the  method for constructing
vertical helicoids in $M\times\R$ which we applied in Examples \ref{exam-helicoidsRnxR}--\ref{exam-berger}.
We also give a local characterization of
vertical helicoids whose horizontal sections are totally geodesic.

Let $I\owns 0$ be an open interval in $\R$  and  let
\[
\Gamma_s:M\rightarrow M, \,\,  s\in I,
\]
be a one-parameter group of isometries of $M$ such that
$\Gamma_0$ is the identity map.
Choose a hypersurface $\Sigma_0^{n-1}\subset M^n,$ define
$\Sigma_s^{n-1}\subset M^n$ by
\[
\Sigma_s=\Gamma_s(\Sigma_0), \,\, s\in I,
\]
and
let  $\eta$ and $\eta_s=\Gamma_{s_{*}}\eta$ be  unit normal fields on
$\Sigma$ and $\Sigma_s$, respectively.

\begin{definition}
Given a constant $a>0,$  we call the   hypersurface
\begin{equation}\label{eq-hypasymptoticline}
\Sigma :=\{(\Gamma_s(p), as)\in M\times\R \,;\, p\in\Sigma_0, \, s\in I\}\subset M\times\R
\end{equation}
the $a$-\emph{pitched twisting} of \,$\Sigma_0$ \emph{determined by}
$\{\Gamma_s\,;\, s\in I\}\subset{\rm Isom}\,(M).$
\end{definition}

Given $p\in\Sigma_0$, denote by $\alpha_p$ the  orbit of $p$ in $M$
under the action of $\Gamma_s$, that is,
\[
\alpha_p(s):=\Gamma_s(p)\in\Sigma_s, \,\, s\in I.
\]
Finally, define  the  $\nu$-\emph{function} of  $\Sigma$ as
\begin{equation}\label{eq-nufunction}
\nu(\alpha_p(s), as):=\langle\alpha_p'(s),\eta_s(\alpha_p(s))\rangle, \,\,\,  (\alpha_p(s), as)\in\Sigma.
\end{equation}

\begin{lemma} \label{lem-constructionhelicoid}
Given $a>0,$ let $\Sigma\subset M\times\R$ be the $a$-pitched twisting of
a hypersurface $\Sigma_0\subset M$ determined by a one-parameter
group $\{\Gamma_s\,;\, s\in I\}\subset{\rm Isom}(M).$
Then, $\nabla\xi$ never vanishes on $\Sigma.$ Furthermore, $\nabla\xi$
defines an asymptotic direction on $\Sigma$
if and only if the gradient $\nabla\nu$  of the $\nu$-function of \,$\Sigma$ is
a horizontal field.
If so, $\Sigma$ is a minimal vertical helicoid in $M\times\R$, provided
\,$\Sigma_0$ is minimal in $M,$ and $\nu$ is nonconstant on $\Sigma.$
Under these conditions, the open set
\[
\Sigma'=\{(\alpha_p(s),as)\in\Sigma \,;\, |\nu(\alpha_p(s),as)|>a\}\subset\Sigma\,,
\]
if nonempty, is spacelike and, then, zero mean isocurved in $M\times\R.$
\end{lemma}

\begin{proof}
Given a point $x=(\alpha_p(s),as)\in\Sigma,$
we have that
\begin{equation}\label{eq-decomposition}
T_x\Sigma=T_{\alpha_p(s)}\Sigma_s\oplus{\rm Span}\{\partial_s\}, \,\,\,\, \partial_s=\alpha_p'(s)+a\partial_t.
\end{equation}
Hence, a unit normal field $N$ for $\Sigma$ in $T(M\times\R)$ can be defined as
\[
N(x):=\frac{-a\eta_s(\alpha_p(s))+\nu(x)\partial_t}{\sqrt{a^2+\nu^2(x)}}\,, \,\,\, x=(\alpha_p(s), as)\in\Sigma\,.
\]
In particular, the angle function of $\Sigma$ at $x$ is given by
\begin{equation}\label{eq-thetanu}
\theta(x)=\frac{\nu(x)}{\sqrt{a^2+\nu^2(x)}}\,\cdot
\end{equation}
Hence, $\theta^2\ne 1$, which implies that $\nabla\xi$ never vanishes on $\Sigma.$
Equality \eqref{eq-thetanu} also gives that $\nabla\theta(x)$ is a multiple of $\nabla\nu(x).$ So,
$\nabla\xi$ is an asymptotic direction of $\Sigma$ if and only if $\langle\nabla\nu(x),\partial_t\rangle =0$
for all $x\in\Sigma$. This proves the first part of the statement.

For the second part, we have just to consider Theorem \ref{th-HR=HL} and observe that all the horizontal sections
of $\Sigma$ are isometric to $\Sigma_0$\,. Hence,  they are minimal if
$\Sigma_0$\, is minimal in $M.$ In addition, a direct computation yields
$
2\theta^2-1=\frac{\nu^2-a^2}{\nu^2+a^2}\,,
$
which implies that $\Sigma'$, if nonempty, is spacelike. This finishes the proof.
\end{proof}

Let $\Sigma$ be as in the above lemma. Given
$x=(\alpha_p(s), as)\in\Sigma,$
considering \eqref{eq-decomposition},
we will denote by $\nabla^s\nu(x)$ the component of $\nabla\nu(x)$ on $T_{\alpha_p(s)}\Sigma_s.$
So, on $\Sigma,$ we have
\[
\nabla\nu=\nabla^{s}\nu+\frac{\langle\nabla\nu,\partial_t\rangle}{a}\partial_s\,.
\]

\begin{lemma}  \label{lem-nuindependent-s}
Let $\Sigma$ be as in Lemma \ref{lem-constructionhelicoid}. Assume that its
$\nu$-function is independent of $s,$ i.e., $\langle\nabla\nu,\partial_s\rangle=0$  on $\Sigma.$
Then $\nabla\nu$ is a horizontal field on $\Sigma$ if and only if
\begin{equation}\label{eq-nuindependent}
\langle\nabla^s\nu(x)\,,\alpha_p'(s)\rangle=0 \,\,\,\,\, \forall x=(\alpha_p(s),as)\in\Sigma.
\end{equation}
\end{lemma}
\begin{proof}
Since $0=\langle\nabla\nu,\partial_s\rangle=\langle\nabla\nu,\alpha_p'+a\partial_t\rangle,$
we have
\[
\langle\nabla\nu,\partial_t\rangle=-\frac{1}{a}\langle\nabla\nu,\alpha_p'\rangle=
-\frac{1}{a}\left(\langle\nabla^s\nu,\alpha_p'\rangle+\frac{\langle\nabla\nu,\partial_t\rangle}{a}\|\alpha_p'\|^2\right).
\]
Hence, $\langle\nabla\nu,\partial_t\rangle=0$ if and only if $\langle\nabla^s\nu,\alpha_p'\rangle=0.$
\end{proof}

Recall that the \emph{cone} over a given  hypersurface $\Sigma_0^{n-1}$ of \,$\s^{n}\subset\R^{n+1}$
is the hypersurface $\widehat\Sigma_0$ of $\R^{n+1}$  which is defined as
\[
\widehat\Sigma_0:=\{rp\in\R^{n+1} \,;\, r\in (0,+\infty),\,  p\in \Sigma_0\}.
\]

It is an elementary fact that a hypersurface $\Sigma_0$ is minimal in $\s^n$ if and only if its associated cone
$\widehat\Sigma_0$ is minimal in $\R^{n+1}.$

\begin{lemma}  \label{lem-cone}
Assume  that $\Sigma_0$ is a hypersurface of \,$\s^n$ and
let $\widehat\Sigma_0$ be the cone of \,$\R^{n+1}$ over  $\Sigma_0$.
Assume further that $\{\Gamma_s\,;\, s\in I\}$ is a one-parameter subgroup of the orthogonal group $O(n+1)={\rm Isom}(\s^{n}).$
Given $a>0,$ denote by $\Sigma\subset\s^n\times\R$ (respect. $\widehat\Sigma\subset\R^{n+1}\times\R$) the
$a$-pitched twisting of \,$\Sigma_0$ in \,$\s^n\times\R$ (respect. $\R^{n+1}\times\R$) determined by
$\{\Gamma_s\,;\, s\in I\}$, that is,
\begin{itemize}[parsep=1ex]
  \item $\Sigma :=\{(\Gamma_s(p), as)\in \s^{n}\times\R \,;\, p\in\Sigma_0, \, s\in I\}\subset \s^n\times\R.$

  \item $\widehat\Sigma :=\{(\Gamma_s(rp), as)\in \R^{n+1}\times\R \,;\,  rp\in\widehat\Sigma_0, \, s\in I\}\subset \R^{n+1}\times\R.$
\end{itemize}
Under these conditions, \,$\Sigma$ is a vertical helicoid in $\s^n\times\R$ if and only if
\,$\widehat\Sigma$ is a vertical helicoid in $\R^{n+1}\times\R.$ Moreover, open
spacelike subsets occur in $\Sigma$ if and only if they occur in $\widehat\Sigma$.
\end{lemma}
\begin{proof}
Setting $x=(\Gamma_s(p), as)\in\Sigma$ and $\hat x =(\Gamma_s(rp), as)\in\widehat\Sigma$,
it is easily checked that  the unit normals
$N(x)\in T\Sigma^\perp$ and $\widehat N(\hat x)\in T\widehat\Sigma^\perp$ coincide (as vectors in $\R^{n+2}$).
Thus, denoting by $\widehat\theta$ the angle function of $\widehat\Sigma,$ we have that
$\theta(x)=\widehat\theta(\hat x)$. Consequently,
$\nabla\theta$ is horizontal on $\Sigma$ if and only if $\nabla\widehat\theta$ is horizontal on $\widehat\Sigma.$
In addition, any horizontal section $\widehat\Sigma_t\subset\R^{n+1}\times\{t\}$ is clearly the cone of
$\Sigma_t\subset\s^n\times\{t\}$ in $\R^{n+1}\times\{t\}.$ In particular,
$\Sigma_t$ is minimal in $\s^n\times\{t\}$  if and only if  $\widehat\Sigma_t$ is minimal in $\R^{n+1}\times\{t\}.$
 Therefore,
$\Sigma$ is a vertical helicoid in $\s^n\times\R$ if and only if
$\widehat\Sigma$ is a vertical helicoid in $\R^{n+1}\times\R$.

The last assertion in the statement  follows from the equality
$\theta(x)=\widehat\theta(\hat x).$
\end{proof}

Now, by means of  Lemmas \ref{lem-constructionhelicoid}--\ref{lem-cone},  we construct
properly embedded vertical helicoids in
$Q_c^n\times\R$ whose horizontal sections project on  totally geodesic hypersurfaces of $Q_c^n.$
First, we handle the Euclidean case $c=0.$
For that, consider the  matrices
\[
J=\left[
          \begin{array}{cc}
            0 & -1 \\
            1 & \phantom{-}0 \\
          \end{array}
        \right] \quad\text{and}\quad
e^{(ks)J}=\left[
          \begin{array}{cc}
            \cos(ks) & -\sin(ks) \\
            \sin(ks) & \phantom{-}\cos(ks) \\
          \end{array}
        \right], \,\, s\in\R,
\]
and, for $k>0,$  define  $\Gamma_s=\Gamma_s(k)$
as the following $n\times n$ block diagonal matrix:

\vt

\begin{itemize}[parsep=2ex]
\item $\Gamma_s=\left[
  \begin{array}{ccccc}
    e^{(ks)J} &  &  & & \\
     &   e^{(ks)J} & & & \\
       &  &  \ddots & &\\
     &  &  &  e^{(ks)J} &\\
     &  &  &  & e^{(ks)J}
  \end{array}
\right]$  ($n$ {even}).
\item
$\Gamma_s=\left[
  \begin{array}{ccccc}
    e^{(ks)J} &  &  & & \\
     &   e^{(ks)J} & & & \\
       &  &  \ddots & & \\
     &  &  &  e^{(ks)J} & \\
     &  &  &   & \,1
  \end{array}
\right]$  ($n$ {odd}).
\end{itemize}

We have that $\mathscr G:=\{\Gamma_s\,;\, s\in\R\}$ is a one-parameter group of isometries of $\R^n.$
So, given $a>0,$ we can choose a totally geodesic hyperplane $\Sigma_0^{n-1}\subset \R^n$ through
the origin  $\boldsymbol 0\in\R^n$ and consider the $a$-pitched twisting
$\Sigma=\Sigma(a,k)$  determined by $\mathscr G$.
In this setting, since $J$ and $e^{(ks)J}$ commute, we  have that
\[\frac{d}{ds}e^{(ks)J}=kJe^{(ks)J}=ke^{(ks)J}J.\]
Hence, for any  $(\Gamma_s(p),as)\in\Sigma,$
\[
\alpha_p'(s):=\frac{d}{ds}\Gamma_s(p)=k\Gamma_s\boldsymbol{J}p,
\]
where
\begin{itemize}[parsep=2ex]
\item $\boldsymbol{J}=\left[
  \begin{array}{ccccc}
    J &  &  &  &\\
     &   J & &  &\\
       &  &  \ddots & &\\
     &  &  &  J & \\
     &  &  &     & J
  \end{array}
\right]$  ($n$ {even}).
\item
$\boldsymbol{J}=\left[
  \begin{array}{ccccc}
   J &  &  & & \\
     &   J & & & \\
       &  &  \ddots & & \\
     &  &  &  J & \\
     &  &  &   & 0
  \end{array}
\right]$  ($n$ {odd}).
\end{itemize}

\vt

Thus,
\begin{equation}\label{eq-nuconstructionhelicoids}
\nu(\alpha_p(s),as)=\langle\alpha_p'(s),\eta_s(\alpha_p(s))\rangle=
k\langle\Gamma_s\boldsymbol{J}p,\Gamma_s\eta(p)\rangle=k\langle\boldsymbol{J}p,\eta(p)\rangle,
\end{equation}
i.e., $\nu$ is nonconstant and independent of $s.$
Also, the  orbits $\alpha_p(s)=\Gamma_s(p)$, $p\in\R^n,$ lie on geodesic spheres of $\R^n$ centered at the origin $\boldsymbol 0.$
Thus, since the hypersurfaces
$\Gamma_s(\Sigma_0)\subset\R^n$ all intersect these  spheres  orthogonally, we have, in particular,
that  \eqref{eq-nuindependent} holds. So, by Lemma \ref{lem-nuindependent-s},  $\nabla\nu$ is horizontal on $\Sigma.$

Now, Lemma \ref{lem-constructionhelicoid} applies
and gives that $\Sigma$ is a properly embedded vertical helicoid in $\R^n\times\R.$ In addition, equality
\eqref{eq-nuconstructionhelicoids} and  the second part of Lemma  \ref{lem-constructionhelicoid}
imply that, for a sufficiently large $k,$ $\Sigma$ contains open spacelike zero mean isocurved subsets.

The above method can be easily adapted for constructing properly embedded vertical helicoids in $\h^n\times\R.$ Indeed,
one has just to consider the standard isometric immersion of $\h^n$ into the Lorentz space $\mathbb{L}^{n+1},$
and then  define the isometries $\Gamma_s$ as

\vt

\begin{itemize}[parsep=2ex]
\item $\Gamma_s=\left[
  \begin{array}{cccccc}
    e^{(ks)J} &  &  & & &\\
     &   e^{(ks)J} & & & &\\
       &  &  \ddots & & & \\
       &  &  & e^{(ks)J} &  \\
     &  &  &  & e^{(ks)J} & \\
     &  &  &  &  & 1
  \end{array}
\right]$  ($n$ {even}).
\item
$\Gamma_s=\left[
  \begin{array}{cccccc}
    e^{(ks)J} &  &  & & &\\
     &   e^{(ks)J} & & & &\\
       &  &  \ddots & & & \\
     &  &  &  e^{(ks)J} & & \\
     &  &  &  & 1\phantom{12} & \\
     &  &  &  &  & 1
  \end{array}
\right]$  ($n$ {odd}).
\end{itemize}
\vt
The rest of the argument is  the same as in the Euclidean case.

For the spherical case $c=1,$  we consider the standard isometric immersion of $\s^n$ into $\R^{n+1}=\R^n\times\R,$
and then define  $\Sigma_0$ as the totally geodesic sphere $\s^n\cap\widehat\Sigma_0,$
where $\widehat\Sigma_0$ is an arbitrary  totally geodesic hyperplane of $\R^{n+1}$ through the origin $\boldsymbol 0$.
For $a,k>0$, the $a$-twisting of $\widehat\Sigma_0$ determined by $\Gamma_s(k)\in{\rm Isom}(\R^{n+1}),$
as described above, is a vertical helicoid in $\R^{n+1}\times\R.$
Since $\widehat\Sigma_0-\{\boldsymbol 0\}$ is the cone of $\R^{n+1}$ over $\Sigma_0\,,$
Lemma \ref{lem-cone} gives that the corresponding
$a$-twisting of $\Sigma_0$ is a properly embedded minimal vertical helicoid in $\s^n\times\R.$

We summarize these considerations in the following

\begin{theorem} \label{th-helicoidsQnXR}
There exists a two-parameter family $\{\Sigma(a,k)\,;\, a, k>0\}$ of
properly embedded vertical helicoids 
in $Q_c^n\times\R$  whose horizontal sections
are vertical translations of totally geodesic hypersurfaces of $Q_c^n.$
Such a $\Sigma(a,k)$ is  an $a$-pitched twisting of a totally geodesic hypersurface \,$\Sigma_0\subset Q_c^n$
determined by a suitable
one-parameter subgroup $\mathscr G=\{\Gamma_s=\Gamma_s(k)\,;\, s\in\R\}$ of \,${\rm Isom}(Q_c^n)$. Furthermore, for any fixed $a>0,$
the parameter $k$ can be chosen  in such a way that $\Sigma(a,k)$ contains
open spacelike zero mean isocurved subsets.
\end{theorem}

Our next result shows that any vertical helicoid in $M\times\R$ with nonvanishing angle function
and totally geodesic horizontal sections is locally a twisting. In particular, Theorem \ref{th-helicoidsQnXR}
admits a local converse.

\begin{theorem} \label{th-localcharacterizationshelicoids}
Let $\Sigma\subset M\times\R$ be a vertical helicoid with nonvanishing angle function. Assume that
each horizontal section $\Sigma_t\subset\Sigma$ is totally geodesic in $M\times\{t\}.$
Then, given $x_0\in\Sigma,$  there exist  a connected open set \,$\Sigma'\owns x_0$ of \,$\Sigma$,
a hypersurface $\mathfrak L_0\subset\pi_{\scriptscriptstyle M}(\Sigma')\subset M,$ and a one-parameter group of isometries
\[
\Gamma_t: \pi_{\scriptscriptstyle M}(\Sigma')\rightarrow\Gamma_t(\pi_{\scriptscriptstyle M}(\Sigma'))\subset M\,,
\,\,\, t\in(-\epsilon,\epsilon),
\]
such that $\Sigma'$ is the  $1$-pitched twisting of $\mathfrak{L}_0$ determined by
$\{\Gamma_t \,;\, t\in(-\epsilon,\epsilon)\},$ that is,
\[
\Sigma'=\{(\Gamma_t(p),t)\in\Sigma \,;\, p\in\mathfrak{L}_0\,, \, t\in(-\epsilon,\epsilon)\}.
\]
\end{theorem}
\begin{proof}
Let  $\varphi_t$ be the flow of the
field $Z=\nabla\xi/\|\nabla\xi\|^2$ on $\Sigma,$ i.e.,
\[
\frac{d\varphi_t}{dt}(x)=Z(\varphi_t(x)) \,\, \forall x\in\Sigma.
\]
Considering that
\[
\frac{d}{dt}\xi(\varphi_t(x))=\left\langle\nabla\xi(\varphi_t(x)),\frac{d\varphi_t(x)}{dt}\right\rangle=1,
\]
we have  $\xi(\varphi_t(x))=t+\xi(x).$ In particular, $\varphi_{t}$ takes
a horizontal section $\Sigma_s$ to $\Sigma_{s+t}$\,.

Since we are assuming $\theta\ne 0,$ we have that \,$\Sigma$
is locally a vertical graph. So, there exists a connected open set
$\Sigma'\owns x_0$ of $\Sigma$ satisfying
$\Sigma'={\rm graph}(u),$ where $u$
is a differentiable function defined on the domain $\Omega=\pi_{\scriptscriptstyle M}(\Sigma')\subset M.$

After a  vertical translation, we can assume  $\Sigma'\cap(M\times\{0\})$
nonempty and  $\pi_{\scriptscriptstyle \R}(\Sigma')=(-2\epsilon, 2\epsilon)$ for some $\epsilon >0.$
In this setting, define the field $Z_0\in T(\Omega)$ as
\[
Z_0(\pi_{\scriptscriptstyle M}(x)):=\pi_{\scriptscriptstyle M_*}Z(x), \,\,\, x\in\Sigma',
\]
and let $\Gamma_t$ be the its flow on $\Omega$, that is,
\[
\Gamma_t(\pi_{\scriptscriptstyle M}(x)):=\pi_{\scriptscriptstyle M}\varphi_t(x), \,\,\, x\in\Sigma'.
\]

Writing $\mathfrak{L}_t:=u^{-1}(t),$
$t\in (-2\epsilon, 2\epsilon),$ one has $\Gamma_t(\mathfrak L_s)=\mathfrak L_{s+t}$ for $|s+t|<2\epsilon.$
(Here, we are identifying $M\times\{0\}$ with $M.$)
Moreover, it follows from \eqref{eq-normaltograph} that $\pi_{\scriptscriptstyle M_*}\nabla\xi$ is parallel to $\nabla u,$
which implies that
$Z_0$ is orthogonal to all level sets $\mathfrak L_t$\,, $t\in (-2\epsilon,2\epsilon ).$

Noticing that the family $\{\mathfrak L_t\,,\, t\in(-2\epsilon,2\epsilon )\}$
defines a totally geodesic foliation of \,$\Omega\subset M,$  we conclude from
\cite[Corollary 6.6]{tondeur} that,
for  $t, s\in(-\epsilon, \epsilon),$
the restriction of $\Gamma_t$ to $\mathfrak{L}_s$ is an isometry over its image
$\Gamma_t(\mathfrak{L}_s)=\mathfrak{L}_{s+t}$\,.   Also, since $\Sigma$ is a vertical helicoid,
we have that $\|\nabla\xi\|$, and so $\|Z\|,$  is constant along the curves $t\mapsto \varphi_t(x), \, x\in\Sigma'$
(see Remark \ref{rem-helices}).
In addition, $Z_0=Z-\langle Z,\partial_t\rangle\partial_t=Z-\partial_t$\,,
and  $\Gamma_{t_*}\circ Z_0=Z_0\circ\Gamma_{t}.$ Thus, for any $p=\pi_{\scriptscriptstyle M}(x),$
$x\in\Sigma',$ we have
\[
\|\Gamma_{t_*}Z_0(p)\|^2 = \|Z_0(\Gamma_t(p))\|^2=\|Z(\varphi_t(x))\|^2-1= {\|Z(x)\|^2-1}=\|Z_0(p)\|^2.
\]

It follows from the above considerations that,  defining $\Omega_\epsilon\subset\Omega$ as the union of all level sets
$\mathfrak{L}_t$ with $t\in(-\epsilon, \epsilon),$ any map
$
p\in\Omega_\epsilon\mapsto \Gamma_t(p), \,\, t\in(-\epsilon,\epsilon),
$
is an isometry  from $\Omega_\epsilon$ to $\Gamma_t(\Omega_\epsilon)\subset\Omega.$
Therefore, if we set,  by abuse of notation,
$\Sigma'=\pi_{\scriptscriptstyle M}^{-1}(\Omega_\epsilon)\cap\Sigma',$ and
$\Omega=\Omega_\epsilon,$
we have that
\[
\Sigma'=\{(\Gamma_t(p),t)\in\Sigma \,;\, p\in\mathfrak{L}_0\,, \, t\in(-\epsilon,\epsilon)\},
\]
as we wished to prove.
\end{proof}

Since $1$-dimensional minimal submanifolds
are totally geodesic, Theorem \ref{th-localcharacterizationshelicoids}  has the following consequence.

\begin{corollary}
Any two-dimensional vertical helicoid $\Sigma^2\subset M^2\times\R$ with nonvanishing angle
function is given, locally, by a twisting of a geodesic of $M.$
\end{corollary}

As a further  application of Lemma \ref{lem-constructionhelicoid},  we now generalize the construction made
in Example \ref{exam-helicoidsS3xR}. Namely, we will obtain  a family of properly embedded vertical
helicoids in the product $\s^{2n+1}\times\R$ by twisting  $2n$-dimensional Clifford tori.

We will adopt the following notation. The identity matrix of order $n+1$ will be denoted by ${\rm Id}$.
We will write $\boldsymbol J$, now, for the $(2n+2)\times(2n+2)$ block matrix
\[
\boldsymbol{J}:=
\left[
\begin{array}{cc}
0 & -{\rm Id}\\
{\rm Id} & \phantom{-}0
\end{array}
\right].
\]
Then, setting $C(t)=(\cos t) {\rm Id}$, and $S(t)=(\sin t) {\rm Id},$ the following identity holds:
\[
e^{t\boldsymbol{J}}=\left[
\begin{array}{cr}
C(t) & -S(t)\\
S(t) & C(t)
\end{array}
\right].
\]
In particular, the derivative of the map $t\in\R\mapsto e^{t\boldsymbol{J}}\in O(2n+2)$ is
\[
\frac{d}{dt}e^{t\boldsymbol{J}}=\boldsymbol{J}e^{t\boldsymbol{J}}.
\]

\begin{theorem} \label{th-twistedcliffordtorus}
Let $\Sigma_0=\s^n(1/\sqrt{2})\times\s^n(1/\sqrt{2})$ be the minimal Clifford torus of the sphere
\,$\s^{2n+1}$. Then, for any $a, k>0,$ the $a$-pitched twisting
\[
\Sigma=\Sigma(a,k):=\{(e^{(ks)\boldsymbol{J}}p,as) \,;\, p\in\Sigma_0, \, s\in\R\}\subset\s^{2n+1}\times\R
\]
is a properly embedded vertical helicoid in $\s^{2n+1}\times\R$. Furthermore, for any fixed $a>0,$
the parameter $k$ can be chosen  in such a way that $\Sigma(a,k)$ contains open spacelike zero mean isocurved subsets.
\end{theorem}
\begin{proof}
Consider the standard immersion of $\s^{2n+1}\times\R$ into $\R^{2n+2}\times\R$ and define the
following local parametrization of $\Sigma$:
\[
\Psi(x_1\,, \dots ,x_n\,, y_1\,, \dots ,y_n,s)=\left(\frac{1}{\sqrt{2}}\Gamma_s((\varphi(x_1\,, \dots, x_n), \psi(y_1\,, \dots ,y_n)),as\right),
\]
where $\Gamma_s=e^{(ks){\boldsymbol{J}}}$ and $\varphi, \psi\colon\R^n\rightarrow\s^n$ are
conformal parametrizations of \,$\s^n.$

Setting $\varphi_i=\partial\varphi/\partial x_i$ and
$\psi_i=\partial\psi/\partial y_i$, we have that
\[
\frac{\partial\Psi}{\partial x_i}=\frac{1}{\sqrt{2}}(\Gamma_s(\varphi_i,0),0)\quad\text{and}\quad
\frac{\partial\Psi}{\partial y_i}=\frac{1}{\sqrt{2}}(\Gamma_s(0,\psi_i),0), \,\,\, 1\le i\le n.
\]
In particular, $\eta_s=\Gamma_s\eta$ is a unit normal field on
$\Sigma_s=\Gamma_s\Sigma_0\subset\s^{2n+1},$ where
\[
\eta=\frac{1}{\sqrt{2}}(\varphi,-\psi).
\]

Writing $x=(x_1\,, \dots ,x_n)$ and  $y=(y_1\,,\dots ,y_n)$,
we have that the orbit of a point
$p=\frac{1}{\sqrt{2}}(\varphi(x), \psi(y))\in\Sigma_0$ under the action of $\Gamma_s$  is
\[
\alpha_p(s)=\Gamma_s(p)=\frac{1}{\sqrt{2}}\Gamma_s(\varphi(x), \psi(y)).
\]
From $\frac{d\Gamma_s}{ds}=k{\boldsymbol{J}}e^{(ks)\boldsymbol{J}}=k\boldsymbol J\Gamma_s=k\Gamma_s\boldsymbol J$, one has
\begin{equation}\label{eq-alphaplinha}
\alpha_p'(s)=\frac{d}{ds}\Gamma_s(p)=k\Gamma_s\boldsymbol Jp=\frac{k}{\sqrt{2}}\Gamma_s(-\psi(y),\varphi(x)).
\end{equation}
Thus, with the notation of Lemma \ref{lem-constructionhelicoid},
\[
\nu(\alpha_p(s),as)=\langle\alpha_p'(s),\eta_s(\alpha_p(s)\rangle=
k\langle\Gamma_s\boldsymbol{J}p,\Gamma_s\eta(p)\rangle=k\langle\boldsymbol{J}p,\eta(p)\rangle=-k\langle\varphi(x),\psi(y)\rangle,
\]
so that $\nu$ is independent of $s.$ Now, for  $i=1,\dots ,n,$ define
\[
a_i:=\frac{\partial\nu}{\partial x_i}=-k\langle\varphi_i,\psi\rangle \quad\text{and}\quad
b_i:=\frac{\partial\nu}{\partial y_i}=-k\langle\varphi,\psi_i\rangle,
\]
and notice that
\[
\|\psi\|^2=\sum_{i=1}^{n}\frac{\langle\psi,\varphi_i\rangle^2}{\|\varphi_i\|^2} =\frac{1}{k^2}\sum_{i=1}^{n}\frac{a_i^2}{\|\varphi_i\|^2}
\quad \text{and} \quad
\|\varphi\|^2=\sum_{i=1}^{n}\frac{\langle\psi_i,\varphi\rangle^2}{\|\psi_i\|^2}=\frac{1}{ k^2}\sum_{i=1}^{n}\frac{b_i^2}{\|\psi_i\|^2}\,\cdot
\]
Hence, setting $\lambda=\langle\varphi_i,\varphi_i\rangle$ and $\mu=\langle\psi_i,\psi_i\rangle$,
$i=1,\dots ,n,$ (recall that $\varphi$ and $\psi$ are both conformal), we have that
\begin{equation}\label{eq-aibi}
\sum_{i=1}^{n}a_i^2=\lambda k^2 \quad\text{and}\quad \sum_{i=1}^{n}b_i^2=\mu k^2,
\end{equation}
for $\|\varphi\|^2=\|\psi\|^2=1.$

From \eqref{eq-alphaplinha}, we have that  ${\partial\Psi}/{\partial s}=\left(\frac{k}{\sqrt{2}}\Gamma_s(-\psi,\varphi),a\right).$ So,
\[
\left\langle\frac{\partial\Psi}{\partial x_i}, \frac{\partial\Psi}{\partial s}\right\rangle=-\frac{k}{2}\langle\varphi_i,\psi\rangle=\frac{a_i}{2}
\quad\text{and}\quad
\left\langle\frac{\partial\Psi}{\partial y_i}, \frac{\partial\Psi}{\partial s}\right\rangle=\frac{k}{2}\langle\varphi,\psi_i\rangle=-\frac{b_i}{2}\,,
\]
from which we conclude  that the $[g_{ij}]$ matrix of $\Sigma$ with respect to $\Psi$ is
\[
[g_{ij}]=\frac{1}{2}\left[
\begin{array}{ccccccc}
\lambda &  &  & &&&  \phantom{-}a_1   \\
       &\ddots & &&&    &      \phantom{-}\vdots    \\
       &       & \lambda &     & & &   \phantom{-}a_n \\
        &             &         & \phantom{-}\mu &    &    &       -b_1 \\
        &             &         &     & \ddots  &    & \phantom{-}\vdots            \\
        &             &         &     &         &  \phantom{-}\mu  & -b_n     \\
    a_1 & \cdots      & a_n     & -b_1 & \cdots  &  -b_n    &        \phantom{-}k^2+2a^2
\end{array}
\right],
\]
where the non dotted missing entries are all zero.

Computing   the cofactors
of the first $2n$ entries of the last line of $[g_{ij}]$, we
conclude that the first $2n$ entries of the last line of $[g^{ij}]=[g_{ij}]^{-1}$ are
\begin{equation} \label{eq-cofactors}
-\frac{\lambda^{n-1}\mu^{n}}{2^{2n}\mathcal{D}}a_1\,, \dots ,-\frac{\lambda^{n-1}\mu^{n}}{2^{2n}\mathcal{D}}a_n\,,
\frac{\lambda^{n}\mu^{n-1}}{2^{2n}\mathcal{D}}b_1\,, \dots ,\frac{\lambda^{n}\mu^{n-1}}{2^{2n}\mathcal{D}}b_n\,,
\end{equation}
where $\mathcal{D}=\det[g_{ij}].$
Since the coordinates of $\nabla\nu$ with respect to the frame
\[\mathfrak B:=\left\{\frac{\partial\Psi}{\partial x_1}\,, \cdots ,\frac{\partial\Psi}{\partial x_n}\,,
\frac{\partial\Psi}{\partial y_1}\,, \cdots ,\frac{\partial\Psi}{\partial y_n}, \frac{\partial\Psi}{\partial s}\right\}\subset T\Sigma\]
are the entries of the column  matrix
\[
[g^{ij}]
\left[
\begin{array}{c}
\frac{\partial\nu}{\partial x_1}\\[.5ex]
\vdots\\[.5ex]
\frac{\partial\nu}{\partial x_n}\\[1ex]
\frac{\partial\nu}{\partial y_1}\\[.5ex]
\vdots\\[.5ex]
\frac{\partial\nu}{\partial y_n}\\[1ex]
\frac{\partial\nu}{\partial s}
\end{array}
\right]=
[g^{ij}]
\left[
\begin{array}{c}
a_1\\[.5ex]
\vdots\\[.5ex]
a_n\\[1ex]
b_1\\[.5ex]
\vdots\\[.5ex]
b_n\\[1ex]
0
\end{array}
\right],
\]
it follows from \eqref{eq-aibi} and \eqref{eq-cofactors} that the last coordinate  of $\nabla\nu$
with respect to $\mathfrak B$ is
\[
\frac{1}{2^{2n}\mathcal{D}}\left(-\lambda^{n-1}\mu^{n}\sum_{i=1}^{n}a_i^2+\lambda^{n}\mu^{n-1}\sum_{i=1}^{n}b_i^2 \right)=
\frac{k^2}{2^{2n}\mathcal{D}}(-\lambda^{n}\mu^{n}+\lambda^{n}\mu^{n})=0,
\]
so that $\nabla\nu$ is a horizontal field on $\Sigma.$

Finally, we observe that
\[|\nu(\alpha_p(s),as)|=k|\langle\varphi(x),\psi(y)\rangle| \,\,\forall (\alpha_p(s),as)\in\Sigma.\]
Thus,  given $a>0$, for  a sufficiently large $k>0,$  the open set of points
of $\Sigma$ on which  $|\nu|>a$ is nonempty. The result, then, follows from Lemma \ref{lem-constructionhelicoid}.
\end{proof}

From the above theorem and Lemma \ref{lem-cone}, we have:

\begin{corollary} \label{prop-twistedcone}
Let $\widehat\Sigma_0\subset\R^{2n+2}$ be the cone over the Clifford torus
$\Sigma_0$ of \,$\s^{2n+1}$. Then, for any $a, k>0,$ the $a$-pitched twisting
\begin{equation}\label{eq-twistedcone}
\widehat\Sigma(a,k):=\{(e^{(ks)\boldsymbol{J}}p,as) \,;\, p\in\widehat\Sigma_0, \, s\in\R\}\subset \R^{2n+2}\times\R
\end{equation}
is an embedded vertical helicoid of $\R^{2n+2}\times\R.$
Furthermore, for any fixed $a>0,$ the parameter $k$ can be chosen  in such a way that $\Sigma(a,k)$ contains
open spacelike zero mean isocurved subsets.
\end{corollary}

It should be mentioned that, through a method different from ours,
Choe and Hoppe \cite{choe-hoppe} showed that the twisted cones
in the above corollary are  minimal hypersurfaces of $\R^{2n+3}$. (We are grateful to Alma Albujer for
let us know about this work.)
A distinguished property of these $a$-twisted cones   is that,
for sufficiently large $a>0,$ they constitute nodal sets of the solutions of the Allen-Cahn differential equation
(see \cite{delpino-musso-pacard}).

\section{Hypersurfaces with a Canonical Direction}  \label{sec-canonicaldirection}

With the aim of introducing and studying vertical catenoids in $M\times\R,$
we proceed now to the characterization of  hypersurfaces of $M\times\R$ which have
$\nabla\xi$ as a principal direction.
Our approach  will be based on the work of R. Tojeiro \cite{tojeiro}, who considered the case
where $M$ is a constant sectional curvature space form $Q_c^n.$

We start with an arbitrary isometric immersion
\[f:\Sigma_0^{n-1}\rightarrow M^n,\]
assuming that there is a neighborhood $\mathscr{U}$
of $\Sigma_0$ in $T\Sigma_0^\perp$  without focal points of $f,$ that is,
the restriction of the normal exponential map $\exp^\perp_{\Sigma_0}:T\Sigma_0^\perp\rightarrow M$ to
$\mathscr{U}$ is a diffeomorphism onto its image. In this case, denoting by
$\eta$ the unit normal field  of $f,$   there is an open interval $I\owns 0$
such that, for all $p\in\Sigma_0,$
\[\gamma_p(s)=\exp_{\scriptscriptstyle M}(f(p),s\eta(p)), \, s\in I,\]
is a well defined geodesic of $M$ without conjugate points. In particular,
for all $s\in I,$
\[
\begin{array}{cccc}
f_s: & \Sigma_0 & \rightarrow & M\\
     &  p       & \mapsto     & \gamma_p(s)
\end{array}
\]
is an immersion of $\Sigma_0$ into $M,$ which is said to be \emph{parallel} to $f.$
Observe that, given $p\in\Sigma_0$, the tangent space $f_{s_*}(T_p\Sigma_0)$ of $f_s$ at $p$ is the parallel transport of $f_{*}(T_p\Sigma_0)$ along
$\gamma_p$ from $0$ to $s.$ Also, with the induced metric,
the unit normal  $\eta_s$  of $f_s$ at $p$ is $\eta_s(p)=\gamma_p'(s).$

Now, define in $M\times\R$ the hypersurface
\begin{equation}\label{eq-paralleldescription}
\Sigma:=\{(f_s(p),a(s))\in M\times\R\,;\, p\in\Sigma_0, \, s\in I\},
\end{equation}
where $a:I\rightarrow a(I)\subset\R$ is an increasing diffeomorphism, i.e., $a'>0.$
We call $\Sigma$ an $(f_s,a)$-\emph{graph} of $M\times\R.$

For any point $x=(f_s(p),a(s))\in\Sigma,$ one has
\[T_x\Sigma=f_{s_*}(T_p\Sigma_0)\oplus {\rm Span}\,\{\partial_s\}, \,\,\, \partial_s=\eta_s+a'(s)\partial_t.\]
A unit normal  to $\Sigma$ is
\[
N=\frac{-a'}{\sqrt{1+(a')^2}}\eta_s+\frac{1}{\sqrt{1+(a')^2}}\partial_t\,.
\]
In particular, its  angle function  is
\begin{equation} \label{eq-thetaparallel}
\theta=\frac{1}{\sqrt{1+(a')^2}}\,\cdot
\end{equation}

\begin{theorem} \label{th-parallel}
If \,$\Sigma$ is an $(f_s,a)$-graph in $M\times\R,$
the following hold:
\begin{itemize}[parsep=1ex]
  \item[{\rm i)}] $\theta$ and $\nabla\xi$ never vanish on \,$\Sigma.$
  \item[{\rm ii)}] $\nabla\xi$ is a principal direction of \,$\Sigma.$
  \item[{\rm iii)}] $\theta$ and the principal curvature of \,$\Sigma$ in
the direction $\nabla\xi$ are constant along the horizontal sections $\Sigma_t$ of \,$\Sigma.$
\end{itemize}
Conversely, if \,$\Sigma\subset M\times\R$
is a hypersurface with nonvanishing angle function which
has $\nabla\xi$ as a principal direction,  then $\Sigma$
is locally an $(f_s,a)$-graph.
\end{theorem}

\begin{proof}
Assume that $\Sigma$ is an $(f_s,a)$-graph of $M\times\R.$
Then, by \eqref{eq-thetaparallel}, $\theta\ne 0$ and $\theta^2\ne 1.$ In particular,
$\nabla\xi$ never vanishes on $\Sigma.$

Since, for any $p\in\Sigma_0$, $\gamma_p$ is a geodesic
of $M$ (and so of $M\times\R$), and $\eta_s=\gamma_p'(s),$  we  have $\overbar\nabla_{\partial_s}\eta_s=0.$
Then, noticing  that
$N=\theta(-a'\eta_s+\partial_t),$ one has
\[
\overbar\nabla_{\partial_s}N=\overbar\nabla_{\partial_s}\theta(-a'\eta_s+\partial_t)=
\frac{\theta'}{\theta\phantom{'}}N-\theta(a''\eta_s+a'\overbar\nabla_{\partial_s}{\eta_s})=
\frac{\theta'}{\theta\phantom{'}}N-\theta a''\eta_s\,.
\]
Hence, for all $X\in\{\partial_s\}^\perp\subset T\Sigma,$ we have that
$\langle \overbar\nabla_{\partial_s}N,X\rangle =0,$
which implies that $\partial_s$ is a principal direction of $\Sigma.$
In addition, one has
\[\langle A\partial_s,\partial_s\rangle=-\langle\overbar\nabla_{\partial_s}N,\partial_s\rangle=a''\theta.\]
So, the corresponding eigenvalue of $A$ is
\[\lambda:=a''\theta^3=\frac{a''}{\sqrt{(1+(a')^2)^3}}\]
(for $\|\partial_s\|^2=1+(a')^2=1/\theta^2$), which  gives that $\lambda$
is a function of $s$ alone, and so
it is constant along the horizontal sections of $\Sigma.$ By \eqref{eq-thetaparallel}, the same is true
for $\theta.$

Finally, observing that
$\nabla\xi=\partial_t-\theta N=a'\theta^2\partial_s,$
we conclude that $\nabla\xi$ is also a principal direction of $\Sigma$
with principal curvature  $\lambda=a''\theta^3,$ i.e.,
\begin{equation}\label{eq-lambda}
A\nabla\xi=(a''\theta^3)\nabla\xi.
\end{equation}
This proves the first part of the theorem.

Conversely,  let us suppose that $\Sigma\subset M\times\R$ is a hypersurface  which
has $\nabla\xi$ as a principal direction and whose angle function $\theta$ never vanishes.
Then, $\Sigma$ is (locally) a graph of a differentiable
function $u$ defined on a domain $\Omega\subset M.$
(By abuse of notation, we keep denoting this local graph by $\Sigma$.)

As we have seen in Section \ref{sec-graphs}, in this setting, 
\begin{equation}\label{eq-thetaandu}
\theta=\frac{1}{\sqrt{1+\|\nabla u\|^2}}\,, 
\end{equation}
where,  
as before, we are writing $\nabla u$ instead of $\nabla u\circ\pi_{\scriptscriptstyle M}.$

Considering the flow $\varphi_t$ of $\nabla\xi/\|\nabla\xi\|^2$ on $\Sigma,$ and
possibly restricting the domain $\Omega,$
we can assume  that
the  horizontal  sections $\Sigma_t\subset\Sigma$
are all connected and homeomorphic to a certain Riemannian manifold $\Sigma_0.$
In other words, there exists an open interval $I_0\owns 0$ such that the map  
$G:\Sigma_0\times I_0\rightarrow \Sigma\subset M\times\R$ given by
\[
G(p,t)=\varphi_t(p)
\]
is a well defined immersion satisfying $G(\Sigma_0\times\{t\})=\Sigma_t$\,.

Define  the map $f_t:\Sigma_0\rightarrow M$ by
\[f_t=\pi_{\scriptscriptstyle M}G(\cdot , t), \,\, t\in I_0,\]
and observe that each $f_t$ is an immersion whose 
image $f_t(\Sigma_0)$ is a level set of $u.$ In particular, $\nabla u$ is orthogonal to
$f_t$ with respect to the induced metric. Furthermore, since $\nabla\xi$ is a principal direction and
$\nabla\theta=-A\nabla\xi,$ we have that $\theta$ is constant along the horizontal sections $\Sigma_t$
(so, the same is true for $\|\nabla\xi\|,$ since $\|\nabla\xi\|^2+\theta^2=1$).
This, together with  \eqref{eq-thetaandu}, gives that, for each $t\in I_0,$  $\|\nabla u\|$ is  constant
on the level set $f_t(\Sigma_0).$ Consequently, the (normalized) trajectories of $\nabla u$ are geodesics
of $M$ (see \cite[Lemma 1]{tojeiro}).

For a fixed $p\in\Sigma_0,$ let us denote by $\varphi_t'(p)$ the velocity vector of the trajectory
$t\in I_0\mapsto\varphi_t(p)\in\Sigma$
at $t,$  that is,
\[\varphi_t'(p)=\frac{\nabla\xi\phantom{^2}}{\|\nabla\xi\|^2}(\varphi_t(p)).\]
In particular, the curve $\gamma_p(t):=\pi_{\scriptscriptstyle M}\circ\varphi_t(p)$
is tangent to $\nabla u$ and, by the above considerations,
is a geodesic of $M$ (when reparametrized by arclength).
Also, from
\[
\|\varphi_t'(p)\|=\frac{1}{\|\nabla\xi(\varphi_t(p))\|} \quad\text{and}\quad \langle\varphi_t'(p),\partial_t\rangle=1,
\]
we have
$
\gamma_p'=\varphi_t'(p)-\langle\varphi_t'(p),\partial_t\rangle\partial_t=\varphi_t'(p)-\partial_t\,,
$
which yields
\begin{equation}\label{eq-final}
\|\gamma_p'\|=\frac{\sqrt{1-\|\nabla\xi\|^2}}{\|\nabla\xi\|}\cdot
\end{equation}

Let $s=L_p(t)\in I\subset\R$ be the arclength parameter of $\gamma_p$ from an arbitrary point $t_0\in I_0.$
Since $\|\nabla\xi\|$ is a function of $t$ alone, it follows from \eqref{eq-final}
that the same is true for $L_p(t).$
Hence, the function  $a=L_p^{-1}:I\rightarrow I_0$ depends only on $s$ and satisfies $a'>0.$
Writing, by abuse of notation, $\gamma_p=\gamma_p\circ a$, and $f_s=f_{a(s)}$\,, one has that
each $\gamma_p$ is an arclength geodesic of $M,$ so that the immersions $f_s$ are parallel and
$\Sigma$ is the corresponding $(f_s,a)$-graph.  This finishes the proof.
\end{proof}

We get from Theorem \ref{th-parallel} the following result,
which classifies the hypersurfaces of $M\times\R$ whose angle function is constant.
For $M=Q_c^n$, this was done in \cite{manfio-tojeiro, tojeiro}. 

\begin{corollary}
Let $\Sigma$ be a connected hypersurface of $M\times\R$. Then, if the  angle function $\theta$
of \,$\Sigma$ is constant, one of the following holds:
\begin{itemize}[parsep=1ex]
\item[{\rm i)}] $\Sigma$ is an open set of $M\times\{t\},\, t\in\R$.
\item[{\rm ii)}] $\Sigma$ is an open set of a  vertical cylinder over a hypersurface of $M.$
\item[{\rm iii)}] $\Sigma$ is  locally an $(f_s,a)$-graph  with $a'$ constant.
\end{itemize}
Conversely, if one of these possibilities occur, then $\theta$ is constant.
\end{corollary}

\begin{proof}
Suppose that $\theta$ is constant on $\Sigma.$
Clearly, (i) occurs if
$\theta^2=1$, and (ii) occurs if $\theta=0.$
Otherwise, $\nabla\xi\ne 0$ and $\theta\ne 0.$ Since,
$A\nabla\xi=-\nabla\theta=0,$  it follows that
$\nabla\xi$ is a principal direction of $\Sigma.$ Hence, by Theorem \ref{th-parallel},
$\Sigma$ is locally an $(f_s,a)$-graph and, by \eqref{eq-thetaparallel}, $a'$ is constant.

The converse is immediate in cases (i) and (ii). The case (iii) follows directly from
equality \eqref{eq-thetaparallel}.
\end{proof}

An important class of hypersurfaces of $Q_c^n\times\R$
having $\nabla\xi$ as a principal direction are the
\emph{rotation hypersurfaces}, which are those obtained by
the  rotation of a plane curve about an axis
$\{o\}\times\R,$  $o\in Q_c^n.$
Clearly, any horizontal section $\Sigma_t$  of a rotational hypersurface
$\Sigma\subset Q_c^n\times\R$ is contained in a  geodesic sphere with center at $(o,t)\in Q_c^n\times\R.$
Considering this property, we introduce the following notion of
rotational hypersurface in $M\times\R$.

\begin{definition} \label{def-rotationalhypersurface}
A hypersurface  $\Sigma\subset M\times\R$ is called \emph{rotational}, if there exists
a fixed point $o\in M$ such that
any horizontal section $\Sigma_t$  is contained in a  geodesic sphere with center
at $(o,t)\in M^n\times\R.$ If so, we call $\{o\}\times\R$ the \emph{axis} of $\Sigma.$
\end{definition}

\begin{remark} \label{rem-rotational}
Let  $\Sigma\subset M\times\R$ be a rotational hypersurface with no horizontal points and
nonvanishing $\theta.$ Since concentric geodesic spheres constitute a parallel family $\{f_s\}$ of
hypersurfaces of $M,$  under these hypotheses, $\Sigma$ is locally an $(f_s,a)$-graph. Hence,  by
Theorem \ref{th-parallel}, \emph{$\nabla\xi$ is a principal direction of any such rotational $\Sigma.$}
\end{remark}

We introduce  now a special type of family of parallel hypersurfaces which will play a fundamental
role in the sequel.

\begin{definition} \label{def-isoparametric}
We call a family of parallel hypersurfaces $f_s:\Sigma_0\rightarrow M,$ $s\in I,$ \emph{isoparametric} if
$f_s$ has constant mean curvature $H_s$ (depending on $s$) for all $s\in I.$ If so, each  hypersurface
$f_s$ is also called \emph{isoparametric}.
\end{definition}

\begin{example} \label{exam-spaceforms}
It is well known that any totally umbilical hypersurface of  $Q_c^n$  is
isoparametric (see, e.g., \cite{vazquez}).
\end{example}

\begin{example} \label{exam-damekricci}
There are certain Hadamard-Einstein manifolds,
known as \emph{Damek-Ricci spaces}, which have many families
of isoparametric hypersurfaces, including its
geodesic spheres. More specifically,  geodesic spheres (of any radius)
in symmetric Damek–Ricci spaces are  isoparametric
with constant principal curvatures, whereas
geodesic spheres (of small radius) in non-symmetric Damek–Ricci spaces are
isoparametric  with nonconstant principal curvatures. The symmetric Damek-Ricci spaces are completely
classified. They are the hyperbolic space $\h^n,$ the complex hyperbolic space  $\C\h^n,$ the quaternionic
hyperbolic space, and the octonionic hyperbolic plane
(see \cite[Section 6]{vazquez} and the references therein for an account of Damek-Ricci spaces).
\end{example}

\begin{example} \label{exam-Ekapatau}
Let
$\mathbb{E}(k,\uptau)$, $k-4\uptau^2\ne 0,$  be one of the simply
connected  3-ho\-mo\-ge\-ne\-ous manifolds with isometry group of dimension 4:
The products $\h^2\times\R$ and $\s^2\times\R$ ($\uptau=0$), the Heisenberg space ${\rm Nil}_3$ ($k=0, \uptau\ne 0$),
the Berger spheres ($k>0, \uptau\ne 0$), or the universal cover
of the special linear group
${\rm SL}_2(\R)$ 
($k<0,\uptau\ne 0$).
In \cite{vazquez-manzano}, the authors classified all isoparametric hypersurfaces of these spaces,
showing, in particular, that  none of them is spherical.
\end{example}

In our next result, we show that there exist  minimal or  constant mean curvature $(f_s,a)$-graphs
in $M\times\R$ if and only if $M$ has isoparametric hypersurfaces.

\begin{theorem} \label{th-parallelcmc}
Let $\Sigma\subset M\times\R$ be an $(f_s,a)$-graph, $s\in I\subset\R,$
such that \,$f_s$ is isoparametric with constant
mean curvature $H_s$\,. Assume that, for a given constant
$H\in\R,$  the diffeomorphism  $a:I\rightarrow a(I)\subset\R$ is defined by the equality
\begin{equation}\label{eq-parallelcmc}
a(s)=\int_{s_0}^{s}\frac{\varrho(\tau)}{\sqrt{1-\varrho(\tau)^2}}d\tau, \,\, s_0\in I,
\end{equation}
where $y=\varrho(s)$ is a solution of the linear differential equation of first order
\begin{equation}\label{eq-diffequation}
y'=H_sy+H
\end{equation}
satisfying  $0<\varrho(s)<1.$ 
Under these conditions, $\Sigma$ has constant mean curvature $H.$
Conversely, if $\Sigma$ has constant mean curvature $H,$ then
$f_s$ is isoparametric and the function $a(s)$ is necessarily given by \eqref{eq-parallelcmc}
with $\varrho=a'\theta.$
\end{theorem}

\begin{proof}
Let us denote the mean curvature of $\Sigma$ by $H_\Sigma$.
By equalities \eqref{eq-meancurvatureslices} and \eqref{eq-lambda}, we get
$H_{s}=\phi(H_\Sigma-\lambda),$ where
\[
\phi=-(1-\theta^2)^{-1/2}=-(a'\theta)^{-1} \quad\text{and}\quad \lambda=a''\theta^3.
\]
So, we have
$
H_\Sigma=-(a'\theta)H_s+a''\theta^3.
$
However, by  \eqref{eq-thetaparallel}, one has
$(a'\theta)'=a''\theta^3.$ Therefore,
if we set $\zeta=a'\theta,$ we get
\begin{equation}\label{eq-cmc}
\zeta'=H_s\zeta+H_\Sigma \,\,\, \forall s\in I.
\end{equation}

A direct computation gives that $0<\zeta^2=(a')^2/(1+(a')^2)<1,$ and also that
\begin{equation}\label{eq-cmc1}
a'=\frac{\zeta}{\theta}=\frac{\zeta}{\sqrt{1-\zeta^2}}\,\cdot
\end{equation}

Thus, if $f_s$ is isoparametric and
the function $a(s)$ is defined by \eqref{eq-parallelcmc} (with $\varrho$ satisfying \eqref{eq-diffequation}),
it follows  by 
\eqref{eq-cmc1} that $\zeta=\varrho.$ Then, comparing  \eqref{eq-diffequation} and \eqref{eq-cmc},
we conclude that $\Sigma$ has constant mean curvature $H.$

Conversely, if $\Sigma$ has constant mean curvature $H_\Sigma=H\in\R,$ it follows from
\eqref{eq-cmc} that $f_s$ is isoparametric and, by \eqref{eq-cmc1}, that $a(s)$ is given by
equality \eqref{eq-parallelcmc} with $\varrho=\zeta=a'\theta.$
\end{proof}

\section{Vertical Catenoids in $M\times\R.$} \label{sec-catenoids}

In this section,
we   introduce the minimal hypersurfaces of $M\times\R$
which resemble the standard catenoids of $\R^3$
with respect to some of its fundamental properties. The definition is as follows.

\begin{definition}
We say that a hypersurface $\Sigma$ of $M\times\R$ with no horizontal points and nonconstant angle
function is a \emph{vertical catenoid} if the following conditions are satisfied:
\begin{itemize}[parsep=1ex]
  \item[i)] $\nabla\xi$ is a principal direction of $\Sigma$ with principal curvature $\lambda\ne 0.$
  \item[ii)] Any horizontal section $\Sigma_t\subset\Sigma$ has nonzero constant mean curvature
  (i.e., depending only on $t$) given by
  \[H_{\Sigma_t}=\frac{\lambda}{\sqrt{1-\theta^2}}\,\cdot\]
\end{itemize}
\end{definition}

Regarding  condition (ii) in the above definition, notice that, from Theorem \ref{th-parallel},
for any $\Sigma$ satisfying condition (i), the function ${\lambda}/{\sqrt{1-\theta^2}}$ is
constant along the horizontal sections $\Sigma_t.$ It should also be noticed that a
vertical catenoid, as defined, is not necessarily rotational (see Definition \ref{def-rotationalhypersurface}).
At the end of this section we construct non rotational properly embedded vertical catenoids in $M\times\R$, where
$M$ is a   Hadamard manifold (see Theorems \ref{th-H-horospheretype} and \ref{th-Ftype}).

The result below follows directly from Lemma \ref{lem-horizontalsection}
and Remark \ref{rem-rotational}. It establishes the minimality of catenoids as hypersurfaces of $M\times\R$,
and also the uniqueness of rotational vertical catenoids as minimal rotational hypersurfaces of $M\times\R$.
The latter is a  well known property of the standard catenoids of Euclidean space $\R^3.$

\begin{proposition}
The following assertions on a hypersurface $\Sigma\subset M\times\R$
with no horizontal points and nonconstant angle function hold:
\begin{itemize}[parsep=1ex]
  \item If $\Sigma$ is a vertical catenoid, then $\Sigma$ is minimal.
  \item If $\Sigma$ is rotational and minimal, then $\Sigma$ is a vertical catenoid.
\end{itemize}
\end{proposition}

It follows from Theorems \ref{th-parallel} and \ref{th-parallelcmc} that,
as long as $M$ contains isoparametric hypersurfaces,
there exist  vertical catenoids in $M\times\R$ which are $(f_s,a)$-graphs.
This applies, for instance, to all manifolds $M$ described in Examples \ref{exam-spaceforms}--\ref{exam-Ekapatau}.
In what follows,  we use this fact to construct {properly embedded} vertical catenoids by
``gluing'' pieces of such graphs.

First, recall that $M$ is said to be a \emph{Hadamard manifold} if it is complete, simply connected, and
has non positive sectional curvature. Any Hadamard manifold $M^n$ is  diffeomorphic to $\R^n$ through  the exponential map, so
that, for a given point $o\in M$, and $r>0,$ the geodesic sphere $S_r(o)$ with center at $o$ and radius $r$ is well defined.
We will write $B_r(o)$ for the geodesic ball of $M$ with center at $o\in M$ and radius $r>0$,  and
$\overbar{B_r(o)}$ for its closure in $M.$

\begin{theorem} \label{th-completecatenoids}
Let $M^n$ be a Hadamard manifold whose geodesic spheres are all
isoparametric. Then, there exists a one-parameter family of properly
embedded rotational catenoids in the product $M\times\R$  which are all
homeomorphic to  $\s^{n-1}\times\R$ and
symmetric with respect to  the  horizontal section
$M\times\{0\}\subset M\times\R.$
\end{theorem}

\begin{proof}
Fix $o\in M$ and choose $r>0.$  For each $s\in (r,+\infty),$
let
\[f_s:\s^{n-1}\rightarrow M^n\simeq M^n\times\{0\}\subset M\times\R\]
be the geodesic sphere of $M$
with center at $o\in M$ and radius $s>r.$
Since $M$ is a Hadamard manifold, each immersion $f_s$ is  convex and non totally geodesic.
Hence, taking the ``outward'' unit normal $\eta_s$ of $f_s$\,, we have that the
(constant) mean curvature $H_s$ of $f_s$ is negative. In particular, setting
\[
\rho(s):=\exp\left(\int_r^sH_u du\right), \, s\in (r,+\infty),
\]
we have that $\rho$ is a solution of $y'=H_sy$ which satisfies
$0<\rho(s)<1$  for all $s>r.$

Now, with the purpose of applying Theorem \ref{th-parallelcmc}, we define the function
\[
a(s):=\int_{r}^{s}\frac{\rho(u)}{\sqrt{1-\rho^2(u)}}du,  \,\,  s\in (r,+\infty).
\]
The integral on the right is improper, for
$\rho(s)\rightarrow 1$ as $s\rightarrow r.$ So, we have to prove that $a$ is well defined, i.e.,
that this integral is convergent. For that, notice that $\rho'(s)\rightarrow H_r<0$ as
$s\rightarrow r.$ In particular, there exist $\delta, C>0$ such that
\[
\rho'(s)>-C \,\, \forall s\in (r, r+\delta).
\]
This, and the fact that $\rho$ is decreasing and satisfies $0<\rho(s)<1$ for $s>r,$ gives
\begin{eqnarray}
\int_{r}^{r+\delta}\frac{\rho(\tau)}{\sqrt{1-\rho^2(\tau)}}d\tau & = &\int_{r}^{r+\delta}\frac{\rho'(u)\rho(u)d\tau}{\rho'(u)\sqrt{1-\rho^2(\tau)}}
\le\frac{1}{C}\int_{\rho(r+\delta)}^{\rho(r)}\frac{d\rho}{\sqrt{1-\rho^2}}\nonumber\\
                                                                 & = & \frac{1}{C}(\arcsin(\rho(r))-\arcsin(\rho(r+\delta)))\le\frac{\pi}{2C}\,, \nonumber
\end{eqnarray}
which implies that the function $a$ is well defined, and that $a(s)\rightarrow 0$ as $s\rightarrow r.$
From this and   Theorems \ref{th-parallel} and \ref{th-parallelcmc}, we conclude that
the $(f_s,a)$-graph, which we denote by $\Sigma_r'$, is
a rotational vertical catenoid.
\begin{figure}[htbp]
\includegraphics{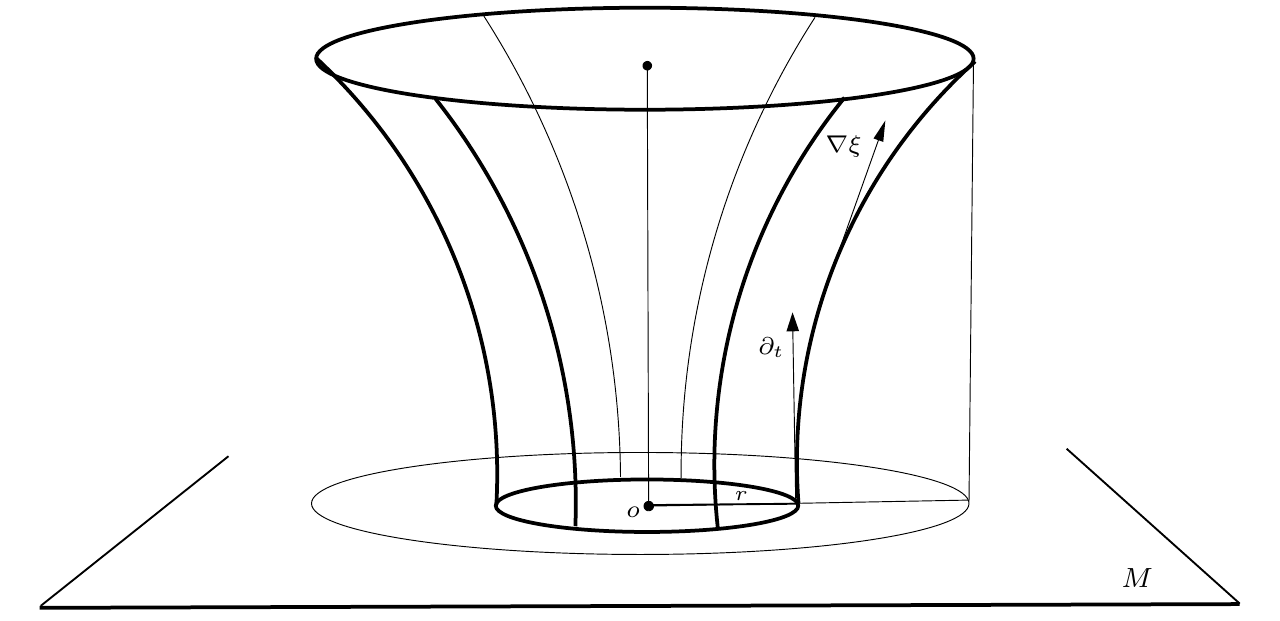}
\caption{The rotational half-catenoid $\Sigma_r'$}
\label{fig-cat}
\end{figure}
Furthermore, $\Sigma_r'$ is clearly a  graph over $M-\overbar{B_r(o)}$
contained in $M\times\R_+$ and with boundary $\partial\Sigma_r'=S_r(o).$ In addition,
\[
a'(s)=\frac{\rho(s)}{\sqrt{1-\rho^2(s)}}\rightarrow +\infty \quad\text{as}\quad s\rightarrow r,
\]
which, together with \eqref{eq-thetaparallel}, gives that $\theta(s)\rightarrow 0$ as $s\rightarrow r.$
Hence,
the tangent spaces  of $\Sigma_r'$ along any  trajectory of $-\nabla\xi$ on $\Sigma_r'$ converge to a vertical
space (i.e., parallel to $\partial_t$) at a point on $\partial\Sigma_r'=S_r(o)$ (see Fig. \ref{fig-cat}).

Now, let $\Sigma_r''\subset M\times\R$ be the reflection of $\Sigma_r'$ with respect to $M\times\{0\}.$ Then,
$\Sigma_r''$ is also a  rotational catenoid in $M\times\R$
with boundary $\partial\Sigma_r''=S_r(o),$  which implies that
it can be ``glued'' together with $\Sigma_r'$ along
$S_r(o),$ that is, we can define
\[
\Sigma_r:={\rm closure}\,(\Sigma_r') \cup {\rm closure}\,(\Sigma_r'').
\]

Since the tangent spaces of $\Sigma_r'$ and $\Sigma_r''$ are vertical along
$S_r(o),$ we have that $\Sigma_r$ is smooth. Moreover, being a geodesic sphere, $S_r(o)$
is a $C^\infty$  manifold. Also,
the trajectories of $\nabla\xi$ on $\Sigma_r$ are geodesics (see \cite[Lemma 1]{tojeiro}) --- so, they are $C^\infty$ as well ---
and any of them intersects $S_r(o)$ transversally. Therefore,
$\Sigma_r$ is a $C^\infty$ properly embedded rotational catenoid in $M\times\R$
which is clearly homeomorphic to $\s^{n-1}\times\R$ and
symmetric with respect to  $M\times\{0\}.$
\end{proof}

The above theorem and
the considerations of  Example \ref{exam-damekricci} give the following result.

\begin{corollary}
Let $M$ be a symmetric Damek-Ricci space. Then, there exists a one-parameter family of properly
embedded rotational catenoids in $M\times\R$  which are  all
homeomorphic to  \,$\s^{n-1}\times\R$ and  symmetric with respect to
$M\times\{0\}.$
\end{corollary}

Assume $M=\R^n$ and let  $\Sigma_r$ be a rotational catenoid as in  Theorem \ref{th-completecatenoids}.
When $n=2$,   $\Sigma_r$ is a standard catenoid of
$\R^3$ obtained by  rotating a catenary  about a fixed axis.
For the half catenoid $\Sigma_r'$ in $\R^n\times\R,$ one has
\[
a(s)=\int_{r}^{s}\frac{r^{n-1}}{\sqrt{\tau^{2n-2}-r^{2n-2}}}d\tau.
\]
It is easily checked that this function is bounded for $n\ge 3.$ So, in this case,
for any $r>0$, the rotational catenoid $\Sigma_r$ is contained in a ``slab'' determined by
two horizontal sections. For $n=2,$ we have
\[a(s)=r\log\left(\frac{s+\sqrt{s^2-r^2}}{r}\right), \,\, s>r,\]
which  is clearly an unbounded function.

In $\h^n\times\R,$  the  height function of any $\Sigma_r$ is \emph{uniformly bounded}.
More precisely, given $n\ge 2,$ for any $r>0,$ $\Sigma_r$ is contained
in a slab of width $\pi/(n-1).$ Indeed, in this setting, the mean curvature of $f_s$ is
$H_s=(1-n)\coth s,$ which gives, for $s\in (r,+\infty),$
\[
\rho(s)=\exp\left(\int_r^sH_\tau d\tau\right)=\exp\left((1-n)\int_r^s\coth\tau d\tau\right)=\left(\frac{\sinh r}{\sinh s}\right)^{n-1}\,\cdot
\]
Thus, the function $a$ which defines $\Sigma_r'$ is
\[
a(s)=\int_{r}^{s}\frac{\rho(\tau)}{\sqrt{1-\rho^2(\tau)}}d\tau=\sinh^{n-1}(r)\int_r^s(\sinh^{2n-2}(\tau)-\sinh^{2n-2}(r))^{-1/2}d\tau.
\]
Applying, in the last integral, the change of variables $v=\sinh\tau/\sinh r,$ we get
\[
a(s)=\sinh r\int_{1}^{\frac{\sinh s}{\sinh r}}(v^{2n-2}-1)^{-1/2}(1+(\sinh^2r)v^2)^{-1/2}dv.
\]
However, $(1+(\sinh^2r)v^2)^{-1/2}<((\sinh r)v)^{-1},$ which implies that
\[
a(s)\le\int_{1}^{\frac{\sinh s}{\sinh r}}\frac{dv}{v\sqrt{v^{2n-2}-1}}=
\frac{1}{n-1}\arctan\sqrt{v^{2n-2}-1}\bigg\rvert_{1}^{\frac{\sinh s}{\sinh r}}\le\frac{\pi}{2(n-1)}\,\cdot
\]

\begin{remark}
In \cite{berard-saearp}, the authors constructed
the rotational catenoids $\Sigma_r$ in $\h^n\times\R$ by rotating suitable curves about an axis. They also
obtained the bound $\pi/2(n-1)$ for the height of the half catenoids $\Sigma_r'.$
\end{remark}

Next, we show that $\s^n\times\R$ admits a  one-parameter family of
rotational catenoids as well.

\begin{theorem}  \label{th-delaunaytype}
There exists a one-parameter family \,$\{\Sigma_r\,;\,0<r<\,{\rm \pi/2}\,\}$
of  properly embedded Delaunay-type rotational catenoids
in $\s^n\times\R$, that is, each $\Sigma_r$ is  periodic, homeomorphic to \,$\s^{n-1}\times\R$,
and has unduloids as the trajectories of the gradient of its  height function.
\end{theorem}

\begin{proof}
Let $f_s:\s^{n-1}\rightarrow\s^n$, $s\in(0,\pi),$ be a family of concentric geodesic spheres of
$\s^n$ with center at $o\in\s^n$ and outward normal orientation, that is, the mean curvature of
$f_s$ is $H_s=-(n-1)\cot(s).$ Given $r\in(0,\pi/2),$ consider
the function
\[
\rho_r(s)=\left(\frac{\sin r}{\sin s}\right)^{n-1}, \,\,\, s\in [r,\pi-r],
\]
which can be verified to be a solution of $y'=H_sy$ satisfying $0<\rho_r|_{(r,\pi-r)}<1.$

Now, let us define the function
\[  
a_r(s)=\int_{r}^{s}\frac{\rho_r(u)}{\sqrt{1-\rho_r^2(u)}}du, \,\,\, s\in (r,\pi-r).
\]  
Since $\rho_r'(r)=H_{r}\ne 0$ and $\rho_r'(\pi-r)=H_{\pi-r}\ne 0,$ we can proceed as in the proof of Theorem \ref{th-completecatenoids}
to conclude that $a_r$ is well defined and  bounded.
In particular, $t_1=a_r(r)$ and  $t_2=a_r(\pi-r)$ are well defined.
\begin{figure}[htbp]
\includegraphics{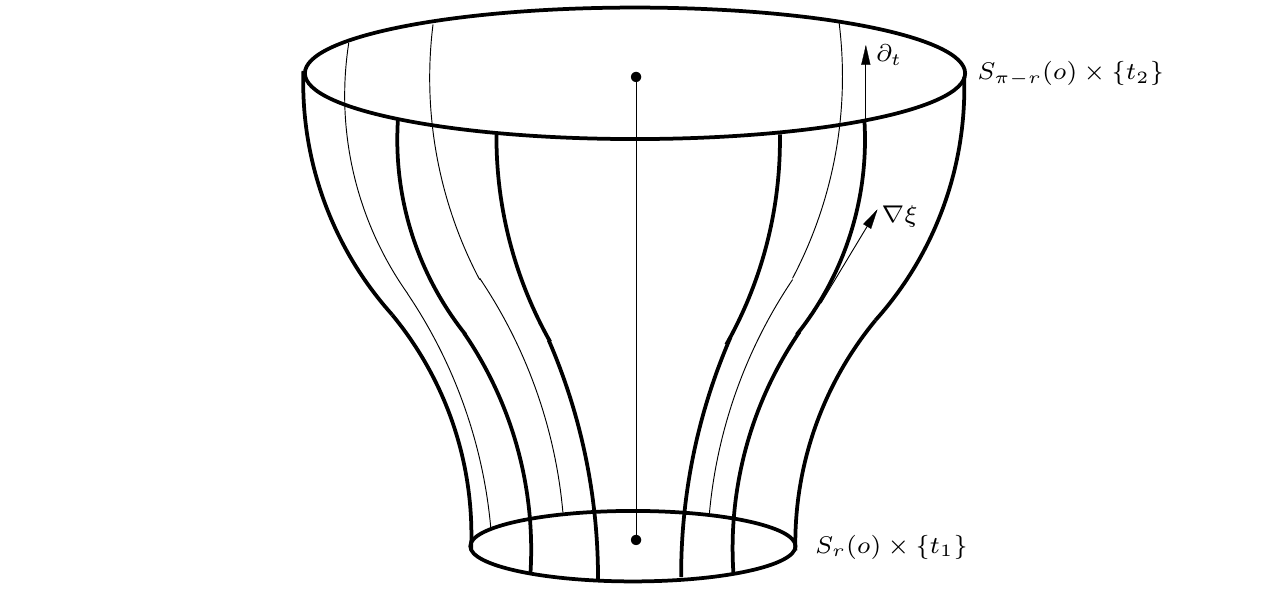}
\caption{\small The ``block'' $\Sigma_r'$ of the rotational catenoid $\Sigma_r$.}
\label{fig-minimaldelaynay}
\end{figure}

It follows from the above that $\Sigma_r'$ is homeomorphic to
$\s^{n-1}\times (r,\pi-r)$ and has boundary
$\partial\Sigma_r'=S_r(o)\times \{t_1\}\cup S_{\pi-r}(o)\times \{t_2\}$
(Fig. \ref{fig-minimaldelaynay}).
Also,
the tangent spaces of $\Sigma_r'$ are vertical along its boundary $\partial\Sigma_r',$
for  $\rho_r(r)=\rho_r(\pi-r)=1.$
Therefore,
from successive  reflections of $\Sigma_r'$
with respect to suitable horizontal sections of $\s^n\times\R$, we obtain a
periodic properly embedded
rotational catenoid $\Sigma_r$  homeomorphic to $\s^{n-1}\times\R.$
\end{proof}

\begin{remark}
The above Delaunay-type catenoids were also obtained in \cite{pedrosa-ritore}.  
\end{remark}

Given a Hadamard manifold $M,$ recall that
the \emph{Busemann function} $\mathfrak b_\gamma(p)$ of $M$ corresponding to an arclength geodesic
$\gamma\colon(-\infty,+\infty)\rightarrow M$ is defined as
\[
\mathfrak b_\gamma(p):=\lim_{s\rightarrow +\infty}({\rm dist}_M(p,\gamma(s))-s), \,\,\, p\in M.
\]

The level sets $\mathscr{H}_s:=\mathfrak b_\gamma^{-1}(s)$ of a Busemann function
$\mathfrak b_\gamma$ are called \emph{horospheres} of $M.$
In this setting, as is well  known, $\{\mathscr H_s\,;\, s\in(-\infty, +\infty)\}$ is a parallel family
which foliates  $M.$ Furthermore,
any geodesic of $M$ which is asymptotic to $\gamma$ (i.e., with  the same point on the asymptotic boundary
$M(\infty)$ of $M$) is orthogonal to each horosphere $\mathscr H_s$\,.
We also remark that horospheres  are submanifolds of class (at least) $C^2$ (see, e.g., \cite[Proposition 3.1]{heintze-hof}).

In hyperbolic space $\h^n,$ any horosphere is totally umbilical with constant principal curvatures
equal to $1.$ Also, as shown  in \cite[Proposition-(vi), pg. 88]{berndtetal}, except for
hyperbolic space\footnote{In \cite{berndtetal}, hyperbolic space is not considered a Damek-Ricci space.},
any Damek-Ricci space contains  a family  $\{\mathscr H_s\,;\, s\in(-\infty, +\infty)\}$ of parallel
horospheres such that  the principal curvatures
of each $\mathscr H_s$ are $1/2$ and $1$, both with constant multiplicities.

Let us see now that, when $M$ is a Hadamard manifold whose horospheres are properly embedded
and  isoparametric with
the same mean curvature, as in the above examples, one can construct properly embedded
vertical catenoids in $M\times\R$ with special properties.

\begin{theorem}  \label{th-H-horospheretype}
Let $\{\mathscr H_s\,;\, s\in(-\infty, +\infty)\}$ be a parallel family of
properly embedded horospheres of constant
mean curvature $H_0>0$ in a Hadamard manifold $M.$ Then,
there exists a properly embedded vertical catenoid $\Sigma$ 
in $M\times\R$ of class at least $C^2$ which is homeomorphic to \,$\R^n$.
Furthermore, $\Sigma$ is foliated by horospheres, is symmetric with respect
to $M\times\{0\},$  and is asymptotic  to both
$M\times\{-\frac{\pi}{2H_0}\}$ and $M\times\{\frac{\pi}{2H_0}\}.$
\end{theorem}
\begin{proof}
For each $s\in (-\infty,\infty)$, consider the isometric immersion
$f_s:\R^{n-1}\rightarrow M^n$  such that $f_s(\R^{n-1})=\mathscr H_s$\,.
Define  the function
\[
\rho(s):=e^{H_0s}, \,\,\, s\in (-\infty,0],
\]
and notice that $\rho$  is a solution of $y'=H_0y$
satisfying
\[
0<\rho(s)< 1=\rho(0) \,\,\, \forall s\in(-\infty, 0).
\]
Thus, by Theorem \ref{th-parallelcmc}, defining
\[
a(s):=\int_{0}^{s}\frac{\rho(u)}{\sqrt{1-\rho^2(u)}}du=\frac{1}{H_0}(\arcsin(e^{H_0s})-\pi/2),
\]
one has that the $(f_s,a)$-graph $\Sigma'$
is a minimal hypersurface of $M\times\R.$ In addition,
\[
\lim_{s\rightarrow-\infty}a(s)=-\frac{\pi }{2H_0}\,\cdot
\]

\begin{figure}[htbp]
\includegraphics{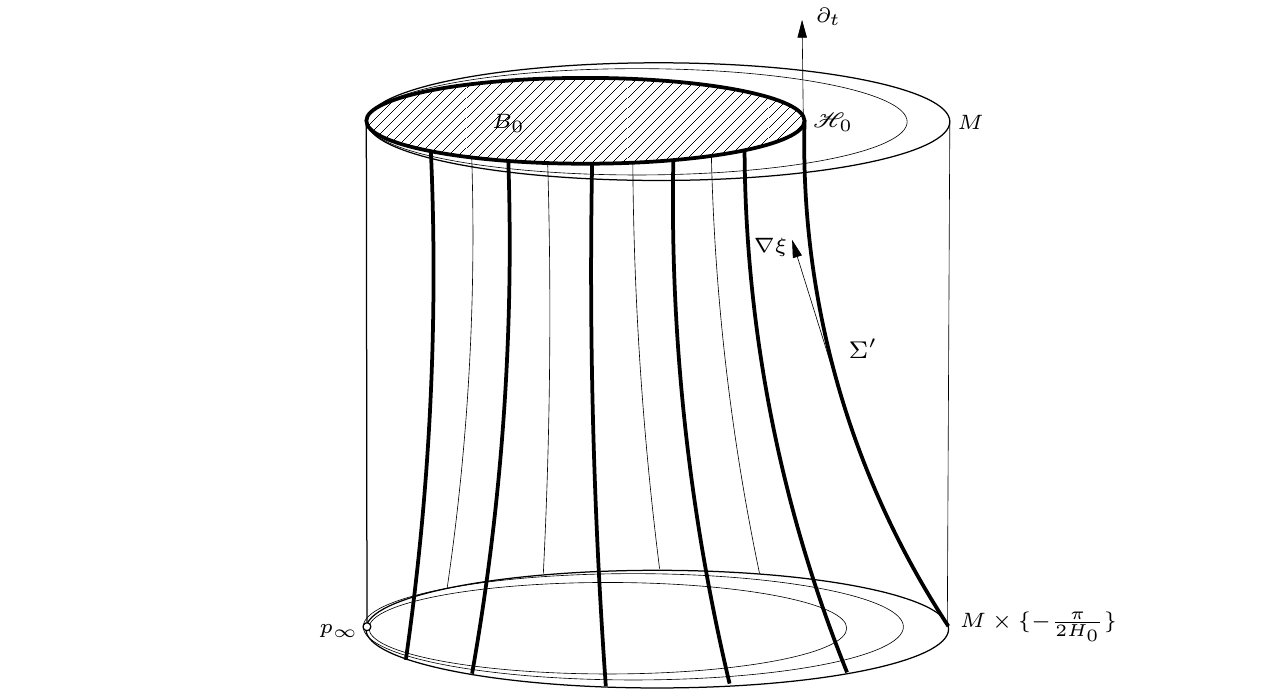}
\caption{\Small The half-catenoid $\Sigma'$ foliated by horospheres.}
\label{fig-Horograph}
\end{figure}

Hence, denoting by $B_0$ the   mean convex side of $\mathscr H_0$\,,
and identifying $M\times\{0\}$ with $M,$
it follows  that $\Sigma'$ is
a minimal graph over $M-B_0$ which has
boundary $\partial\Sigma'=\mathscr H_0$ and
is asymptotic to $M\times\{-\frac{\pi}{2H_0}\}$
(see Fig. \ref{fig-Horograph}).
In particular, $\Sigma'$ is homeomorphic to $\R^n$.

We also have that $\rho(0)=1.$ So, as in the previous theorems,
any trajectory of $\nabla\xi$ on $\Sigma'$ meets  $\partial\Sigma'$ orthogonally.
Therefore, setting $\Sigma''$ for the reflection of $\Sigma'$ with respect to
$M\times\{0\}$, and defining
$
\Sigma:={\rm closure}\,(\Sigma')\cup {\rm closure}\,(\Sigma''),
$
we can argue  just as before
and conclude that $\Sigma$ is a properly embedded  $C^2$-differentiable
(for horospheres are, at least, $C^2$ differentiable)
vertical catenoid  of $M\times\R$ which has all the stated properties.
\end{proof}

In our next result, we consider more general isoparametric foliations of Hadamard manifolds.

\begin{theorem}  \label{th-Ftype}
Let $\mathscr F:=\{f_s:\Sigma_0\rightarrow M,$ $s\in (-\infty,+\infty)\}$ be an
isoparametric family of hypersurfaces in  a Hadamard manifold \,$M^n.$ Assume that:
\begin{itemize}[parsep=1ex]
\item[\rm i)] For all $s\in (-\infty,+\infty),$ $f_s$ is a $C^k$ $(k\ge 2)$  proper embedding
with positive mean curvature $H_s\,.$
  \item [\rm ii)] $\mathscr F$ foliates $M,$   i.e., $M=\bigcup f_s(\Sigma_0),$ \,$s\in (-\infty,+\infty).$
\end{itemize}
Then,  there exists  a properly embedded
$C^k$ catenoid $\Sigma$ in $M\times\R$  which is homeomorphic to
$\Sigma_0\times\R.$ Furthermore, $\Sigma$ is foliated by (vertical translations of)
the leaves of \,$\mathscr F$ and is symmetric with respect to $M\times\{0\}.$
\end{theorem}
\begin{proof}
Since $H_s>0$ for all $s\in (-\infty,+\infty),$ we have that the function
\[
\rho(s):=\exp\left(\int_{0}^{s}H_udu \right), \,\, s\in (-\infty, 0],
\]
which is  a solution of $y'=H_sy,$ satisfies:
\[
0<\rho(s)< 1=\rho(0) \,\,\, \forall s\in (-\infty, 0).
\]
In addition,  $\rho'(0)=a(0)H_0>0.$ From this, as in the preceding  proofs, we get that
\[
a(s):=\int_{0}^{s}\frac{\rho(u)}{\sqrt{1-\rho^2(u)}}du,  \,\,\, s\in(-\infty, 0),
\]
is a well defined function, i.e., this improper integral is convergent.
So, the $(f_s,a)$-graph $\Sigma'$ is a minimal
graph over $M-B_0$ whose $\nabla\xi$-trajectories meet $\partial\Sigma'=\mathfrak L_0\times\{0\}$ orthogonally.
Here, $B_0\subset M$ is the  mean convex side of $\mathfrak L_0$\,.
In particular, $\Sigma'$ is homeomorphic to $\Sigma_0\times\R$.
Now, by reflecting $\Sigma'$ with respect to $M\times\{0\},$
as we did before, we obtain the desired vertical catenoid of $M\times\R.$
\end{proof}

In hyperbolic space $\h^n,$ the well known families  of equidistant hypersurfaces
satisfy the conditions of Theorem \ref{th-Ftype},
since they   foliate $\h^n$ and have constant mean curvature between $0$ and $1.$
Also, each leaf of such a family is $C^\infty$ and homeomorphic to $\R^n.$ So, we have the following final result, which
was obtained in \cite{daniel} for the particular case $n=2.$

\begin{corollary} \label{cor-equidistants}
Let $\mathscr F:=\{f_s:\R^{n-1}\rightarrow \h^n,$ $s\in (-\infty,+\infty)\}$
be a family of parallel equidistant hypersurfaces in $\h^n.$ Then,
there exists a properly embedded $C^{\infty}$  vertical catenoid in
\,$\h^n\times\R$ which is homeomorphic to $\R^n$. Moreover, $\Sigma$ is
symmetric with respect to $\h^n\times\{0\}$  and is foliated by
(vertical translations of) the leaves of  $\mathscr F.$
\end{corollary}

\end{document}